 \documentclass[final,1p,times]{elsarticle}
\usepackage{amsmath,amsthm,amssymb}

\usepackage{lscape}
\usepackage{subfigure}

\usepackage{graphicx}
\usepackage{latexsym,amssymb}
\usepackage{amsmath}

\usepackage{setspace}
\def\vc#1{\mbox{\boldmath $#1$}}

\newcommand{\down}[2]{\smash{\lower#2\hbox{#1}}}
\newcommand{\up}[2]{\smash{\lower-#2\hbox{#1}}}

\newcommand{\bbN}{\mathbb{N}}

\newcommand{\bbZ}{\mathbb{Z}}
\def\ifcompleted{\iftrue}  % If completed.
\def\vc#1{\mbox{\boldmath $#1$}}
\newtheorem{thm}{Theorem}[section]

\newtheorem{lem}[thm]{Lemma}
\newtheorem{prop}[thm]{Proposition}
\newtheorem{remark}{Remark}
\newtheorem{coro}[thm]{Corollary}

\begin{document}

\begin{frontmatter}



%\title{Modelling of Retrial, Abandonment and After-Call Work in Call Centers}
%\title{Performance Modeling for Call Centers with Abandonment and After-Call Work and Redial}
%\title{Performance Analysis of Call Centers with Abandonment and After-Call Work and Redial}
%\title{Modelling After-Call Work, Abandonment  and Redial in Call Centers}
%\title{Modelling Abandonment and Redial in Call Centers with After-Call Work}
%\title{Abandonment and Redial in Call Centers with After-Call Work}

%\title{Two Simple Approaches to M/M/$c$/Setup Queues}
\title{Asymptotic Analysis for Markovian Queues with Two Types of Nonpersistent Retrial Customers}

%\title{A comprehensive performance analysis of call centers with abandonment, retrial and after-call work}


\author[label1]{Tuan Phung-Duc}

\address[label1]{Department of Mathematical and Computing Sciences \\
Tokyo Institute of Technology, Ookayama, Tokyo 152-8552, Japan \\
              Tel.: +81-(0)3-5734-3851 \quad
              Fax: +81-(0)3-5734-2752  \\
E-mail: tuan@is.titech.ac.jp}

\begin{abstract}
We consider Markovian multiserver retrial queues where a blocked customer 
has two opportunities for abandonment: at the moment of blocking or at the departure epoch from the orbit. 
In this queueing system, the number of customers in the system (servers and buffer) and that in the orbit form a 
level-dependent quasi-birth-and-death (QBD) process whose stationary distribution is expressed in 
terms of a sequence of rate matrices. Using a simple perturbation technique and a matrix analytic method, we derive Taylor series expansion 
for nonzero elements of the rate matrices with respect to the number of customers in the orbit. 
We also obtain explicit expressions for all the coefficients of the expansion. Furthermore, we derive tail asymptotic formulae for the joint stationary distribution of the number of customers in the system and that in the orbit. 
Numerical examples reveal that the tail probability of the model with two types of nonpersistent customers is greater than that of the corresponding 
model with one type of nonpersistent customers.
\end{abstract}

\end{frontmatter}

%\keywords{Taylor series expansion \and Perturbation \and Asymptotic analysis \and Multiserver retrial queue \and Level-dependent QBD \and 
%Matrix analytic method \and Censoring}
 %\subclass{68M20 \and 90B22 \and 60K25}

\section{Introduction}
Retrial queues are characterized by the fact that an arriving customer that is blocked leaves the service area but repeats the request after some random time. These models naturally arise from various modelling problems of telecommunication and network systems~\cite{artalejo08}. The reader is referred to~\cite{artalejo_review} for a list of recent papers on retrial queues. Research of retrial queues is pioneered by Cohen~\cite{Cohen56} who proposes and analyzes the multiserver model. Due to the inhomogeneity in the underlying Markov chain, the analysis of multiserver retrial queues is much more difficult than that of corresponding models without retrials. As a result, analytical solutions for multiserver retrial queues have been obtained for only a few special cases. An explicit solution for the joint stationary distribution of the server state and the number of customers in the orbit is obtained only for the M/M/1/1 retrial queue~\cite{Fali97}. For the M/M/2/2 retrial queue, the joint stationary distribution is expressed in terms of hypergeometric functions \cite{Avram11,Fali97,Hans87,phung2}. 
Do~\cite{Do10} presents an analytic solution for an M/M/1/1 retrial queue with working vacation and constant retrial rate. 

We refer to~\cite{Kim95,Choi98,Gome99,Pearce89,phung1,phung2} for effort to find analytical solutions for M/M/$c$/$c$ retrial queues 
with more than two servers by the generating function approach. Kim~\cite{Kim95} and Gomez-Corral and Ramalhoto~\cite{Gome99} deal with 
the case of three servers while Choi and Kim~\cite{Choi98} derive analytical solution for a model with feedback. It should be noted that 
some technical assumptions are imposed in these papers. Using an alternative approach, Phung-Duc et al.~\cite{phung1} show that the joint 
stationary distribution is expressed in terms of continued fractions for the cases of $c=3$ and 4, without any technical 
assumption as presented in literature. The same authors in~\cite{phung2} further derive analytical solutions for the joint stationary distribution 
of state-dependent M/M/$c$/$c+r$ retrial queues with Bernoulli abandonment, where $c+r \leq 4$. 
Pearce~\cite{Pearce89} presents an expression for the joint stationary distribution in terms of generalized continued fractions for the 
M/M/$c$/$c$ retrial queue with any $c$. Although, the formulae in~\cite{Pearce89} do not directly yield a 
numerical algorithm, this is one of the seminal papers providing the most general analytical results for the model. 

Recently, asymptotic analysis for multiserver retrial queues has been receiving considerable attention.
Liu and Zhao~\cite{Binliu10} use a censoring technique and a level-dependent QBD approach 
to derive analytical solutions for the M/M/$c$/$c$ retrial queues for 
the cases of $c=1$ and 2 and investigate the asymptotic behavior for the stationary distribution of the general case 
with any $c$. Using the same approach, Liu et al.~\cite{Binliu11} extend their study to M/M/$c$/$c$ retrial queues 
with one type of nonpersistent customers. Kim et al.~\cite{3Kim2012} derive more detailed asymptotic formulae for the joint stationary 
distribution of M/M/$c$/$c$ retrial queues in comparison with those obtained by Liu and Zhao~\cite{Binliu10}. Furthermore, Kim and 
Kim~\cite{2Kim2012} refine the asymptotic result obtained by Liu et al.~\cite{Binliu11}.
The methodology of~\cite{3Kim2012,2Kim2012} is based on an investigation of the analyticity of generating functions. 
However, the asymptotic formulae presented in~\cite{3Kim2012,2Kim2012} still contain some unknown coefficients.

We recall that the number of customers in the system and that in the 
orbit form a level-dependent QBD process whose stationary distribution can be expressed in terms of a sequence of rate matrices~\cite{Ramaswami_Taylor96}. 
Liu et al.~\cite{Binliu11,Binliu10} focus on the asymptotic behavior of the joint stationary distribution. 
To this end, they derive a few essential expansion formulae (up to three terms) for some elements of the rate matrices, which 
are enough for their purpose. Liu and Zhao~\cite{Binliu10} remark that it seems that there is no unified pattern for the higher order expansions. 
In this paper, motivated by Liu et al.~\cite{Binliu11,Binliu10}, we present an exhaustive 
perturbation analysis for the M/M/$c$/$K$ retrial queue with two types of nonpersistent customers, which is recently introduced and numerically analyzed in ~\cite{phung-duc13}. 
This model extends those with and without one type of nonpersistent customers~\cite{Binliu11,Binliu10}. 

Using a unified simple approach, we are able to derive Taylor series expansion for all nonzero elements of the rate matrices. Furthermore, 
the coefficients of this Taylor series expansion are explicitly obtained in terms of recursive formulae. 
It should be noted that the rate matrix for level-dependent QBD process cannot be obtained in a closed form in general. 
To the best of our knowledge, this paper is the first which obtains Taylor series expansion for the rate matrix of a level-dependent QBD process.
In addition, we also derive tail asymptotic bounds for the joint stationary distribution using the methodology developed by Liu et al.~\cite{Binliu11,Binliu10}.

The rest of the paper is organized as follows. In Section~\ref{ModelDes:sec} we present multiserver retrial queues 
with two types of nonpersistent customers and their level-dependent QBD formulation. Section~\ref{perturbation:sec} is 
devoted to the main results where we present expansion formulae for all the nonzero elements of the rate matrices. Section~\ref{numerical:exam} 
shows some numerical examples to demonstrate the accuracy of the Taylor series expansions and the tail asymptotic bounds of the joint stationary distribution. Section~\ref{conclusion:sec} concludes the paper.

\section{Model Description and Preliminaries}\label{ModelDes:sec}
\begin{figure}[tbhp]
\begin{center}
\includegraphics[scale=0.35]{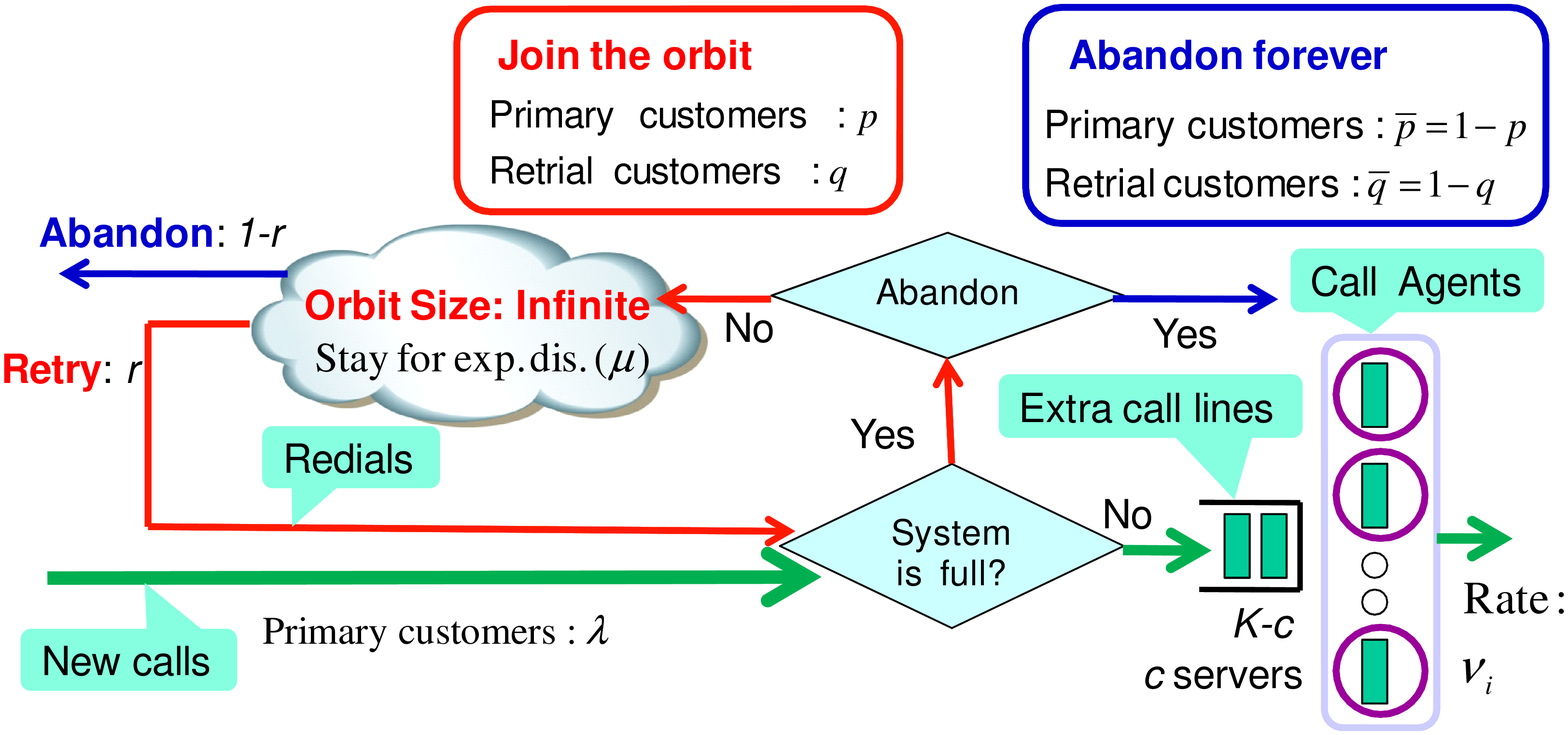} 
\caption{Retrial Queues with Two Types of Nonpersistent Customers.}
\label{model:fig}
\end{center}
\end{figure}

\subsection{Model description}
We first describe the M/M/$c$/$K$ retrial queue, where there are $c$ servers 
and a waiting room with $K -c$ waiting positions in front of the servers. 
Primary customers arrive at the servers according to a Poisson process with rate $\lambda > 0$ and 
the total service rate of all the servers is $\nu_i$, provided that there are $i$ customers 
in the system (the servers and the waiting room). We assume that $0= \nu_0 \leq \nu_1 \leq \nu_2 \leq \cdots \leq \nu_{K-1} \leq \nu_K$.
An arriving primary customer enters the system if possible otherwise the primary customer either moves to the orbit with probability $p$ or 
abandons (not joins the orbit) with probability $\bar{p} = 1-p$.

A customer in the orbit is called a {\it retrial customer} hereafter. 
Each retrial customer stays in the orbit for an exponentially distributed time 
with a finite positive mean $1/\mu$. Upon the departure epoch from the orbit, each customer either 
retries to enter the servers or abandons forever with probabilities $r$ and $\bar{r} = 1-r$, respectively. 
A retrial customer that does not abandon at the departure epoch from the orbit either joins 
the orbit again with probability $q$ or gives up forever with probability $\bar{q}=1-q$ 
if the system is fully occupied upon arrival, otherwise the retrial customer enters the system.

Let $X(t) = (C(t),N(t))$ ($t \ge 0$), where $C(t)$ and $N(t)$ denote
the number of customers in the system and that in the orbit, at time $t$, respectively. 
It is easy to see that the bivariate process $\{X(t);t \ge 0\}$ 
is a Markov chain with the state space $\{0,1,\dots,K \} \times \bbZ_+$, where
$\bbZ_+=\{0,1,2,\dots \}$. See Figure~\ref{model:fig} for the flows of customers.

It is easy to show that if $q<1$ or $r<1$, $\{X(t);t \ge 0\}$ is always ergodic, otherwise the Markov chain is 
ergodic if and only if $\rho = (\lambda p) / \nu_K < 1$ (See e.g.~Phung-Duc et al.~\cite{phung2} or Falin and Templeton~\cite{Fali97}). 
Throughout the paper, we assume that $\{ X(t); t \geq 0\}$ is ergodic.

\subsection{Motivation of the model}
In almost all retrial queue literature, the orbit is an abstracted unit which is not given a clear physical justification. 
In this paper, we introduce the service function into the orbit. To this end, we assume that the orbit provides some 
kind of service and a blocked customer may or may not be satisfied with this service. In case of being satisfied, 
the customer departs from the orbit, otherwise it reattempts to get service at the original service facility. Thus, 
we may consider the orbit as the secondary service of the primary M/M/$c$/$K$ loss system. Furthermore, the model 
presented here generalizes several models studied in the literature~\cite{artalejo08,artalejo_review}.

In particular, our model is motivated from modelling of a call center operating in a group of cooperative call centers. 
In this situation, a blocked call in a call center may be forwarded to another one. 
However, due to the speciality  of each call center, an operator in a call center may not 
be able to answer perfectly a call forwarded from the others. A blocked call either satisfies with the service 
of the forwarded call center and departs or reattempts for service in the original one.  We refer to~\cite{phung-duc13}  for some more 
motivations of the model.

\subsection{Level-dependent QBD formulation}
It is easy to confirm that $\{X(t); t \geq 0\}$ is a level-dependent QBD process whose infinitesimal generator is given by 
\begin{eqnarray*}
\vc{Q}
&=&
\left (
\begin{array}{llllll}
	\vc{Q}^{(0)}_1 \  & \vc{Q}^{(0)}_0  \  & \vc{O}  \  &  \vc{O} \ & \cdots \
	\\
	\vc{Q}^{(1)}_2 & \vc{Q}^{(1)}_1   & \vc{Q}^{(1)}_0 & \vc{O} &  \cdots
	\\
	\vc{O}   & \vc{Q}^{(2)}_2 & \vc{Q}^{(2)}_1 & \vc{Q}^{(2)}_0 &  \cdots
	\\
	\vc{O}   &  \vc{O}    & \vc{Q}^{(3)}_2 & \vc{Q}^{(3)}_1  & \cdots 
	\\
	\vdots   & \vdots     & \vdots   & \vdots  & \ddots
\end{array}
\right ),
\end{eqnarray*}
where $\vc{O}$ denotes a matrix of an appropriate dimension with all zero entries and   
 $\vc{Q}^{(n)}_0$, $\vc{Q}^{(n)}_1$ ($n \in \bbZ_+$) and $\vc{Q}^{(n)}_2$ ($n \in \bbN = \{1,2,\dots \} $) are given by
\begin{eqnarray}
\vc{Q}^{(n)}_0  & = & 
\left (
\begin{array}{cccccc}
	0 & 0 & \cdots &  0 & 0 \\
	0    & 0 & \cdots &  0 & 0 \\
	\vdots & \vdots & \ddots & \vdots & \vdots \\
	0 & 0 &  \cdots &0 & 0 \\
	0 & 0 &  \cdots & 0 & \lambda p  \\
\end{array}
\right ),  
\qquad 
\vc{Q}^{(n)}_2   =  
\left (
\begin{array}{cccccc}
	n \mu \bar{r}  \  & n \mu r  \  & 0 & \cdots &  0 \\
	0 & n \mu \bar{r} & n \mu r \  &  \ddots &  \vdots \\
	\vdots &  & \ddots & \ddots &   0 \\
	\vdots &   &  & n \mu \bar{r} \  & n \mu r  \\
	0 & \cdots & \cdots & 0 & n \mu (\bar{r} + r \bar{q}) \\
\end{array} 
\right ), \qquad \nonumber \\
\vc{Q}^{(n)}_1  & = & 
\left (
\begin{array}{cccccc}
	b^{(n)}_0 \ & \lambda \ & 0 & \cdots  \ &  \cdots & 0 \\
	\nu_1 \  & b^{(n)}_1 \ & \lambda  & \ddots  \ &   & \vdots \\
	0 \ & \nu_2 \ & b^{(n)}_2 \ & \ddots &  \ddots & \vdots \\
	\vdots & \ddots & \ddots & \ddots & \ddots & 0 \\
	\vdots & & \ddots & \ddots \quad & b^{(n)}_{K -1} \quad & \lambda \\
	0 & \cdots & \cdots & 0 &  \nu_K \ & b^{(n)}_K \\
\end{array}
\right ). \nonumber
%\]
\end{eqnarray}
The diagonal elements of $\vc{Q}^{(n)}_1$ are given by 
\begin{eqnarray*}
	b^{(n)}_i & = & -(\lambda   + n \mu  + \nu_i), \qquad i = 0,1,2,\dots, K-1, \\
	b^{(n)}_K & = & - (\lambda p  + n \mu (\bar{r} + r \bar{q}) + \nu_K).
\end{eqnarray*}

Let 
\[
	\pi_{i,n} = \lim_{t \to \infty} \Pr\{C(t)=i, N(t)=n\}, \qquad i = 0,1,2,\dots,K, \qquad n \in \bbZ_+,
\]
denote the joint stationary probability of the number of customers in the system and that in the orbit. 

Furthermore, let $\vc{\pi}_n =(\pi_{0,n},\pi_{1,n},\dots,\pi_{K,n})$ and $\vc{\pi} =(\vc{\pi}_0,\vc{\pi}_1,\dots)$. 
The stationary distribution $\vc{\pi}$ is the solution of the following system of equations.
\begin{eqnarray}
	\label{fund:eq}
	\vc{\pi} \vc{Q} = \vc{0}, \qquad \vc{\pi} \vc{e} = 1, 
\end{eqnarray}
where vectors $\vc{e}$ and $\vc{0}$ denote a column vector and a row vector with an appropriate dimension 
whose entries are ones and zeros, respectively.
Equation (\ref{fund:eq}) is rewritten in a vector form as follows.
\begin{eqnarray}
	\label{vector:recurn}
	\vc{\pi}_{n-1} \vc{Q}^{(n-1)}_{0} + \vc{\vc{\pi}}_{n} \vc{Q}^{(n)}_1 + \vc{\vc{\pi}}_{n+1} \vc{Q}^{(n+1)}_2 
	& = &  \vc{0}, \qquad n \in \bbN, \\
	\label{normalize:eq}
	\vc{\vc{\pi}} \vc{e}  & = &  1.
\end{eqnarray}
The solution of (\ref{vector:recurn}) and (\ref{normalize:eq}) is given by 
\[
       \vc{\vc{\pi}}_{n}  =  \vc{\vc{\pi}}_{n-1} \vc{R}^{(n)}, \qquad n \in \bbN, 
\]
in which $\{\vc{R}^{(n)}; n \in \bbN\}$ is the minimal nonnegative solution of 
\begin{equation}\label{ricachi:eq}
      \vc{Q}^{(n-1)}_{0} + \vc{R}^{(n)} \vc{Q}^{(n)}_1 + \vc{R}^{(n)} \vc{R}^{(n+1)} \vc{Q}^{(n+1)}_2  = \vc{O}, \qquad n \in \bbN,
\end{equation}
according to~\cite{Ramaswami_Taylor96}. Furthermore, the boundary vector $\vc{\vc{\pi}}_0$ is the solution of 
\begin{eqnarray}
     \vc{\vc{\pi}}_0 (\vc{Q}^{(0)}_{1} + \vc{R}^{(1)}  \vc{Q}^{(1)}_{2} )       & = & \vc{0}, \nonumber \\ 
     \vc{\vc{\pi}}_0( \vc{I} + \vc{R}^{(1)} + \vc{R}^{(1)}\vc{R}^{(2)} + \dots) \vc{e}  & = & 1. \nonumber
\end{eqnarray}
where $\vc{I}$ denotes an identity matrix with an appropriate dimension.
%
%
%
%
%
%

%The rationale of the fundamental algorithm is the following two propositions.
The rate matrices can be computed using the algorithm in~\cite{phung-duc10b} based on Propositions~\ref{prop:recursive} and~\ref{prop:rn} below.
\begin{prop}
\label{prop:recursive}
Let ${\mathcal M}$ denote a set of real square matrices of order $K+1$. We define $R_{n} : {\mathcal M} \rightarrow {\mathcal M}$ as
\begin{align*}
R_{n}(\vc{X}) = \vc{Q}^{(n-1)}_{0} \left(- \vc{Q}^{(n)}_{1} -\vc{X} \vc{Q}^{(n+1)}_{2}\right)^{-1}, \qquad  n \in \bbN.
\end{align*}
Then, according to (\ref{ricachi:eq}), the matrix sequence $\{R^{(n)}; n \in \bbN\}$ satisfies the following backward recursive equation.
\begin{align}\label{backward:for}
\vc{R}^{(n)} = R_{n}(\vc{R}^{(n+1)}) = R_{n} \circ R_{n+1} \circ \cdots \circ R_{n+k} \circ \cdots , \qquad n \in \bbN,
\end{align}
where $f \circ g(\cdot) = f(g(\cdot))$.
\end{prop}
Proposition~\ref{prop:recursive} shows that $\vc{R}^{(n)}$ can be viewed as an {\em infinite} matrix continued fraction. 
%It implies that we have to resort to an approximation to evaluate $R^{(n)}$ if we compute $R^{(n)}$ numerically. 
The following proposition provides a sequence of matrices that converges to $R^{(n)}$.
\begin{prop}[Proposition 2.4 in Phung-Duc {\it et al}. (2010)~\cite{phung-duc10b}]
\label{prop:rn}
If we define the matrix sequence $\{\vc{R}^{(n)}_{k}; k \in \bbZ_+ \}$ by
\begin{align*}
	\vc{R}^{(n)}_{0} & = \vc{O}, \qquad k = 0, \\
	\vc{R}^{(n)}_{k} & = R_{n}(\vc{R}^{(n+1)}_{k-1}) = \cdots = R_{n} \circ R_{n+1} \circ \cdots \circ R_{n+k-1}( \vc{O}), \qquad k, n \in \bbN, 
\end{align*}   
%
%where $f \circ g(\cdot) = f(g(\cdot))$, 
then we have
\begin{equation*}
\lim_{k \to \infty}\vc{R}^{(n)}_{k} = \vc{R}^{(n)}, \qquad n \in \bbN.
\end{equation*}
\end{prop}

From Proposition~\ref{prop:recursive}, we see that the first $K$ rows of $\vc{R}^{(n)}$ must be zeros. Let 
\[ \vc{r}^{(n)} = (r^{(n)}_0,r^{(n)}_1, \dots, r^{(n)}_K) \] 
denote the last row of $\vc{R}^{(n)}$. Comparing the last row of in both sides of (\ref{ricachi:eq}) yields
\begin{align}
	\label{i0:eq}
    b^{(n)}_0 r^{(n)}_0   + \nu_1 r^{(n)}_1 + \widetilde{r}^{(n+1)}_{0} r^{(n)}_K & =  0, & i &=0, \\
	\label{i1c_m_1:eq}
    \lambda r^{(n)}_{i-1}  +  b^{(n)}_i r^{(n)}_i + \nu_{i+1} r^{(n)}_{i+1} + \widetilde{r}^{(n+1)}_{i} r^{(n)}_K & =  0, &	i &=1,2,\dots, K-1,\\
	\label{ic:eq}
    \lambda r^{(n)}_{K-1}  +  \left( b^{(n)}_K + \widetilde{r}^{(n+1)}_{K} \right) r^{(n)}_K  & =  - p \lambda, & i & = K,
\end{align}
where 
\begin{align}
	\widetilde{r}^{(n)}_{0} & =   n \mu \bar{r} r^{(n)}_{0}, &  i & = 0, \nonumber \\
	\widetilde{r}^{(n)}_{i} & =   n \mu r r^{(n)}_{i-1} + n \mu \bar{r} r^{(n)}_{i}, & i & = 1,2,\dots,K-1, \nonumber \\
	\widetilde{r}^{(n)}_{K} & =   n \mu r r^{(n)}_{K-1} + n \mu (\bar{r}+r \bar{q}) r^{(n)}_{K}, & i & = K. \nonumber
\end{align}
\begin{prop} \label{sum:coro}
We have 
\[
    r^{(n)}_0 + r^{(n)}_1 + \cdots + r^{(n)}_{K-1 } + (\bar{r} + r \bar{q}) r^{(n)}_K = \frac{\lambda p}{n \mu}.
\]
\end{prop}

\begin{proof}
This proposition follows from the fact that the following matrix represents the infinitesimal generator of 
the ergodic Markov chain $\{X(t); t \geq 0 \}$ censored in levels $\{ l(i); i = 0,1,\dots, n-1\}$, where $l(i) = ((0,i),(1,i),\dots,(K,i))$.
\begin{eqnarray}
\vc{Q}^{\leq n-1}
&=&
\left (
\begin{array}{llllll}
	\vc{Q}^{(0)}_1 \  & \vc{Q}^{(0)}_0  \  & \vc{O}  \  &  \cdots \ & \vc{O} \
	\\
	\vc{Q}^{(1)}_2 & \vc{Q}^{(1)}_1   & \vc{Q}^{(1)}_0 & \ddots &  \vc{O}
	\\
	\vc{O}   & \vc{Q}^{(2)}_2 & \vc{Q}^{(2)}_1 & \ddots &  \vdots
	\\
	\vdots   &  \vc{O}    & \ddots & \ddots  & \vc{O} 
	\\
	\vdots   & \ddots     & \ddots   & \vc{Q}^{(n-2)}_2  & \vc{Q}^{(n-2)}_0	\\
	\vc{O}   & \cdots     & \vc{O}   & \vc{Q}^{(n-1)}_2  & \widehat{\vc{Q}}^{(n-1)}
\end{array}
\right ), \label{sensored:chain}
\end{eqnarray}
where 
\[
	\widehat{\vc{Q}}^{(n-1)}  = \vc{Q}^{(n-1)}_1 + \vc{R}^{(n)} \vc{Q}^{(n)}_2.
\]
Therefore, 
\[
	(\vc{Q}^{(n-1)}_2 + \widehat{\vc{Q}}^{(n-1)}) \vc{e} = \vc{0}. 
\] 
By comparing the last elements of both sides, we obtain the announced result.
\end{proof}
\begin{remark}
Proposition~\ref{sum:coro} is the key for the derivation of the series expansion for $r^{(n)}_i$ ($i = 0,1,\dots,K-1,K$). 
The cases $p = q = r = 1$ and $r = 1$ have been presented as the key lemma in~\cite{Binliu11} and~\cite{Binliu10}, respectively.
\end{remark}

\begin{remark}
Liu et al.~\cite{Binliu11} use (\ref{sum:coro}) and (\ref{i0:eq}) to obtain explicit expressions for $r^{(n)}_0$ and $r^{(n)}_1$ of a single server retrial queue with one type of nonpersistent customer, i.e., $r=1$ and $K=c=1$. 
However, we observe here that if $r \neq 1$, such explicit formulae cannot be obtained. This implies an essential difference between the model with one type of nonpersistent customer~\cite{Binliu11} and ours.
\end{remark}

\section{Main Results}\label{perturbation:sec}
In this section, we present a perturbation analysis for all the elements of the rate matrices. In particular, 
we derive Taylor series expansion for $r^{(n)}_i$ ($i=0,1,\dots,K$) with respect to $1/n$. The case $\bar{r} + r \bar{q} > 0$ and 
the case $q = r = 1$ are essentially different. Thus, we present the former and the latter separately in Section~\ref{first_case:sec} 
and Section~\ref{second_case:sec}, respectively.

%Thus, we derive perturbation formulae for these two cases separately.
%
\subsection{The case $\bar{r} + r \bar{q} > 0$}\label{first_case:sec}
In this section, we explain the process for obtaining Taylor series expansion for $r^{(n)}_{K-k}$ ($k=0,1,\dots,K$) 
in details. We drive the first three terms in the expansion of $r^{(n)}_{K-k}$ step by step before going to the general result.
In what follows, we use the sequences $\{ \gamma^{(k)}_n; n \in \bbZ_+ \}$ where $k$ implies the number of idle servers with 
the convention that $\gamma^{(k)}_n = 0$ if $k > K$. We define $o(x)$ as $\lim_{x \to 0} o(x)/x = 0$. 
\subsubsection{One term expansion}
\begin{lem}\label{eva_rate_vector1:lem}
We have $\lim_{n \to \infty} n^{k+1} r^{(n)}_i = 0$ for $i=0,1,\dots,K-k-1$ and 
\begin{equation}\label{first_expand:eq}
	r^{(n)}_{K-k} = \gamma^{(k)}_1 \frac{1}{n^{k+1}} + o(\frac{1}{n^{k+1}}), \qquad k=0,1,\dots,K,
\end{equation}
where 
\[
   \gamma^{(0)}_1 = \frac{\lambda p}{\mu(\bar{r} + r \bar{q})}, \qquad 
   \gamma^{(k)}_1 = \frac{\nu_{K-k+1}}{\mu} \gamma^{(k-1)}_1, \qquad k=1,2,\dots,K.
\]
\end{lem}
\begin{proof}
We prove Lemma~\ref{eva_rate_vector1:lem} using mathematical induction. First of all, we show that Lemma~\ref{eva_rate_vector1:lem} is true for $k=0$.
Indeed, from Proposition~\ref{sum:coro}, we see that $\lim_{n \to \infty} r^{(n)}_k = 0$ ($k=0,1,\dots,K$).
It follows from (\ref{i0:eq}) that $\lim_{n \to \infty}  n r^{(n)}_0 = 0$. Equation (\ref{i1c_m_1:eq}) with 
$i = 1,2,\dots,K-1$ yields $\lim_{n \to \infty}  n r^{(n)}_i = 0$ ($k=0,1,\dots,K-1$). 
From $\lim_{n \to \infty}  n r^{(n)}_{K-1} = 0$ and (\ref{ic:eq}), we can show that 
\[ 
	r^{(n)}_K = \frac{\lambda p}{\mu (\bar{r} + r \bar{q})} \frac{1}{n} + o(\frac{1}{n}),
\]  
implying that $\gamma^{(0)}_1 = \frac{\lambda p}{\mu (\bar{r} + r \bar{q})}$ and that Lemma~\ref{eva_rate_vector1:lem} is true 
for $k=0$.

We assume that Lemma~\ref{eva_rate_vector1:lem} is true for $k-1$, i.e., $\lim_{n \to \infty} n^k r^{(n)}_i = 0$ ($i = 0,1,\dots,K-k$) and 
\[
	r^{(n)}_{K-(k-1)} = \gamma^{(k-1)}_1 \frac{1}{n^k} + o (\frac{1}{n^k}), 
\]
for some $k =1,2,\dots,K$. We will prove that Lemma~\ref{eva_rate_vector1:lem} is also true for $k$.

Indeed, we multiply $n^k$ by (\ref{i0:eq}) and (\ref{i1c_m_1:eq}) ($i=1,2,\dots,K-k-1$) and use  
the fact that $\lim_{n \to \infty} n^k r^{(n)}_i = 0$ ($i = 0,1,\dots,K-k$) in order to 
obtain $\lim_{n \to \infty} n^{k+1} r^{(n)}_i = 0$ ($i=0,1,\dots,K-k-1$). We are going to 
derive the first term in the expansion of $r^{(n)}_{K-k}$. Transforming (\ref{i0:eq}) and (\ref{i1c_m_1:eq}) 
($i=1,2,\dots,K-1$) yields
\begin{equation}\label{rK_k:eq}
	r^{(n)}_{K-k} = \frac{\lambda r^{(n)}_{K-k-1}}{n \mu} + \frac{\nu_{K-k+1} r^{(n)}_{K-k+1}}{n \mu} + \frac{\widetilde{r}^{(n+1)}_{K-k} r^{(n)}_K}{n \mu} - \frac{\lambda + \nu_{K-k}}{n \mu} r^{(n)}_{K-k},
\end{equation}
for $k=1,2,\dots,K$, where $r^{(n)}_{-1} = 0$. Equation (\ref{rK_k:eq}) plays a key role in our derivation in this section.

Using the assumptions of mathematical induction leads to
\begin{eqnarray*}
	\frac{\lambda r^{(n)}_{K-k-1}}{n \mu} & = & \frac{\lambda r^{(n)}_{K-k-1} n^k}{n^{k+1} \mu} = o(\frac{1}{n^{k+1}}), \\
	\frac{\nu_{K-k+1} r^{(n)}_{K-k+1}}{n \mu} & = & \frac{\nu_{K-k+1}}{\mu} \gamma^{(k-1)}_1 \frac{1}{n^{k+1}} + o (\frac{1}{n^{k+1}}), \\
	\frac{\widetilde{r}^{(n+1)}_{K-k} r^{(n)}_K}{n \mu} & = & \frac{n+1}{n} \frac{r n^k r^{(n+1)}_{K-k-1}  + \bar{r} n^k r^{(n+1)}_{K-k}  }{n^k} \frac{n r^{(n)}_K}{n} = o(\frac{1}{n^{k+1}}),\\
	\frac{\lambda + \nu_{K-k}}{n \mu} r^{(n)}_{K-k} & = & \frac{\lambda + \nu_{K-k}}{n^{k+1} \mu} r^{(n)}_{K-k} n^k = o(\frac{1}{n^{k+1}}).	
\end{eqnarray*}
Substituting these formulae into (\ref{rK_k:eq}) yields the announced result.
\end{proof}

\begin{remark}
Our derivations in Section~\ref{first_case:sec} are based on mathematical induction which has two steps. The first step is the formula for $r^{(n)}_{K-0}$. The second step is to derive the formula for $r^{(n)}_{K-k}$ provided that for $r^{(n)}_{K-i}$ ($i=0,1,\dots,k-1$) is true.
Proposition~\ref{sum:coro} is important for the first step while formula (\ref{rK_k:eq}) is the key for the second step.
\end{remark}

Lemma~\ref{eva_rate_vector1:lem} can be refined as follows. 
\begin{lem}\label{first_term:lem}
We have
\begin{equation}\label{first_term_refine:eq}
	r^{(n)}_{K-k} = \gamma^{(k)}_1 \frac{1}{n^{k+1}}  + O(\frac{1}{n^{k+2}}),
\end{equation}
where $O(x)$ denotes $\lim_{x \to 0} |O(x)/x| = C \geq 0$.
\end{lem}
\begin{proof}
We prove Lemma~\ref{first_term:lem} using mathematical induction. Indeed, it follows from Proposition~\ref{sum:coro} that
\[
	r^{(n)}_K = \frac{\lambda p}{(\bar{r} + r \bar{q})\mu} \frac{1}{n} - \frac{1}{\bar{r} + r \bar{q}}  \sum_{i=0}^{K-1} r^{(n)}_i.
\]
Using this equation and Lemma~\ref{eva_rate_vector1:lem}, we obtain (\ref{first_term_refine:eq}) with $k=0$.
Assuming that Lemma~\ref{first_term:lem} is true for $r^{(n)}_{K-i}$ ($i= 0,1,\dots,k-1$), we prove that Lemma~\ref{eva_rate_vector1:lem} is true for $r^{(n)}_{K-k}$. 
%Lemma~\ref{first_term:lem} is also true for $r^{(n)}_{K-k}$. 

%
From (\ref{first_expand:eq}), we see that the first, the third and the forth terms in the left hand side (\ref{rK_k:eq}) is 
in order $O(1/n^{k+2})$. Furthermore, from (\ref{rK_k:eq}) and the assumption of mathematical induction  
\[ 
    r^{(n)}_{K-k+1} = \gamma^{(k-1)}_{1} \frac{1}{n^k} + O(\frac{1}{n^{k+1}}),   
\]
we obtain %(\ref{first_term_refine:eq}).
\[
	r^{(n)}_{K-k} = \gamma^{(k-1)}_{1} \frac{\nu_{K-k+1}}{\mu} \frac{1}{n^{k+1}} + O(\frac{1}{n^{k+2}}),
\]
which implies (\ref{first_term_refine:eq}).
\end{proof}

\subsubsection{Two term expansion}
\begin{lem}\label{second_term_expansion:lem}
\begin{equation}\label{second_term_expansion:eq}
	r^{(n)}_{K-k} = \gamma^{(k)}_1 \frac{1}{n^{k+1}} - \gamma^{(k)}_2 \frac{1}{n^{k+2}} + O(\frac{1}{n^{k+3}}), 
\end{equation}
where 
\begin{eqnarray}
	\label{two_term_expand:initial}
	\gamma^{(0)}_2 & = & \frac{\nu_K \lambda p}{\mu^2 (\bar{r} + r \bar{q})^2}, \nonumber \\
	\gamma^{(k)}_2 & = & \frac{\nu_{K-k+1}}{\mu} \gamma^{(k-1)}_2 + \left( \frac{\lambda + \nu_{K-k}}{\mu} 
						- \frac{\lambda p \bar{r}}{\mu (\bar{r} + r \bar{q})} \right) \gamma^{(k)}_1, \qquad k = 1,2,\dots,K. \qquad
\end{eqnarray}
\end{lem}
\begin{proof}
It follows from Proposition~\ref{sum:coro} that
\begin{eqnarray}\label{rK_other:eq}
	r^{(n)}_K & = & \frac{\lambda p}{\mu(\bar{r} + r \bar{q})n} - \frac{1}{\bar{r} + r \bar{q}} \sum_{i=0}^{K-1} r^{(n)}_i \nonumber \\ 
             & =  & \frac{\lambda p}{\mu(\bar{r} + r \bar{q})} \frac{1}{n} - \frac{\nu_K \lambda p}{\mu^2 (\bar{r} + r \bar{q})^2} \frac{1}{n^2} + O(\frac{1}{n^3}),
\end{eqnarray}
implying the announced result with $k=0$. 
We derive $\gamma^{(k)}_2$ ($k=1,2,\dots,K$). Assuming that (\ref{second_term_expansion:eq}) is 
true for $r^{(n)}_{K-i}$ ($i=0,1,\dots,k-1$), we prove that (\ref{second_term_expansion:eq}) is also true for $r^{(n)}_{K-k}$. 

We again look at the right hand side of (\ref{rK_k:eq}). Using the assumption of mathematical induction, we have one term expansion for $r^{(n+1)}_{K-k-1}$, $r^{(n+1)}_{K-k}$ and 
$r^{(n)}_{K-k}$ using (\ref{first_term_refine:eq}) and two term expansion for $r^{(n)}_{K-k+1}$ as follows.
\begin{eqnarray}
   \label{rnKkm1:f_termeq}
   r^{(n)}_{K-k-1} & = & \gamma^{(k+1)}_1 \frac{1}{n^{k+2}} + O (\frac{1}{n^{k+3}}), \\
   \label{rnKkm2:f_termeq}
   r^{(n)}_{K-k}   & = & \gamma^{(k)}_1 \frac{1}{n^{k+1}} + O (\frac{1}{n^{k+2}}), \\
   \label{rnKkm3:f_termeq}
   r^{(n)}_{K-k+1}   & = & \gamma^{(k-1)}_1 \frac{1}{n^{k}} - \gamma^{(k-1)}_2 \frac{1}{n^{k+1}} + O (\frac{1}{n^{k+2}}).    
\end{eqnarray}
Furthermore, we have
\[
	\frac{\widetilde{r}^{(n+1)}_{K-k}}{\mu}  =   \bar{r} \gamma^{(k)}_1 \frac{1}{(n+1)^k} + O (\frac{1}{(n+1)^{k+1}}) = \bar{r} \gamma^{(k)}_1 \frac{1}{n^k} + O (\frac{1}{n^{k+1}}),
\]
and 
\[
	\frac{r^{(n)}_K}{n}  =  \gamma^{(0)}_1 \frac{1}{n^2} + O (\frac{1}{n^3}), 
\]
leading to
\begin{equation}
	\label{complex_term:f_termeq}
	\frac{\widetilde{r}^{(n+1)}_{K-k}  r^{(n)}_K}{n \mu} = \bar{r} \gamma^{(0)}_1 \gamma^{(k)}_1 \frac{1}{n^{k+2}} + O (\frac{1}{n^{k+3}}).
\end{equation}
Substituting (\ref{rnKkm1:f_termeq})--(\ref{complex_term:f_termeq}) into the right hand side of (\ref{rK_k:eq}) and arranging the results yields
\begin{eqnarray*}
	r^{(n)}_{K-k} & = & \gamma^{(k)}_1 \frac{1}{n^{k+1}} - \left[ \frac{\nu_{K-k+1}}{\mu} \gamma^{(k-1)}_2  + 
                     \left( \frac{\lambda + \nu_{K-k}}{\mu} - \gamma^{(0)}_1 \bar{r}   \right) \gamma^{(k)}_1   \right] \frac{1}{n^{k+2}} \nonumber \\
                  &   & \mbox{} + O (\frac{1}{n^{k+3}}),
\end{eqnarray*}
which implies (\ref{two_term_expand:initial}).
\end{proof}
\subsubsection{Three term expansion}

It follows from Proposition~\ref{sum:coro} that 
\[
	r^{(n)}_K = \gamma^{(0)}_1 \frac{1}{n} - \frac{1}{\bar{r} + r \bar{q}} \sum_{k=1}^K r^{(n)}_{K-k}. 
\]
We substitute (\ref{second_term_expansion:eq}) into this equation in order to obtain
\begin{equation}
\label{rnK:eq}
	r^{(n)}_K = \gamma^{(0)}_1 \frac{1}{n} - \gamma^{(0)}_2 \frac{1}{n^2} + \gamma^{(0)}_3 \frac{1}{n^3} + O (\frac{1}{n^4}), 
\end{equation}
where 
\begin{eqnarray*}
	%\gamma^{(0)}_3 & = & \frac{\gamma^{(1)}_2 - \gamma^{(2)}_1 }{\bar{r} + r \bar{q}} = \frac{\lambda p \nu_K^2}{\mu^3(\bar{r}+r\bar{q})^3} \left(1 + \frac{\lambda (\bar{r}+r\bar{q})}{\nu_K} - \frac{\lambda \bar{r} p}{\nu_K} \right).
	\gamma^{(0)}_3 & = & \frac{\gamma^{(1)}_2 - \gamma^{(2)}_1 }{\bar{r} + r \bar{q}}.
\end{eqnarray*}
Equation (\ref{rnK:eq}) suggests that we can find three term expansion for $r^{(n)}_K$. The following lemma shows this property for $r^{(n)}_{K-k}$ ($k=1,2,\dots,K$). 
\begin{lem}\label{third_term_expansion:lem}
We improve the expansion for $r^{(n)}_{K-k}$ ($k=0,1,\dots,K$) as follows.
\begin{equation}
  r^{(n)}_{K-k} = \gamma^{(k)}_1 \frac{1}{n^{k+1}} - \gamma^{(k)}_2 \frac{1}{n^{k+2}} + \gamma^{(k)}_3 \frac{1}{n^{k+3}} + O(\frac{1}{n^{k+4}}), 
\end{equation}
where $\gamma^{(k)}_3$ ($k=1,2,\dots,K$) is calculated using the following recursion. 
\begin{eqnarray*}
    \gamma^{(0)}_3 & = & \frac{\gamma^{(1)}_2 - \gamma^{(2)}_1 }{\bar{r} + r \bar{q}}, \\
	\gamma^{(k)}_3 & = & \frac{\nu_{K-k+1}}{\mu} \gamma^{(k-1)}_3 + \left( \frac{\lambda}{\mu} + \frac{\lambda p r}{\mu(\bar{r} + r \bar{q})} \right) \gamma^{(k+1)}_1 \\  
                   &   & \mbox{} +  \left(  \frac{\lambda + \nu_{K-k}}{\mu} - \frac{\lambda p \bar{r}}{\mu(\bar{r} + r \bar{q})} \right) \gamma^{(k)}_2 - \frac{\lambda p \bar{r} [k\mu (\bar{r} + r \bar{q}) + \nu_K] }{\mu^2 (\bar{r} + r \bar{q})^2} \gamma^{(k)}_1.
\end{eqnarray*}
\end{lem}
\begin{proof}
We again give a proof for Lemma~\ref{third_term_expansion:lem} using mathematical induction.
Equation (\ref{rnK:eq}) implies Lemma~\ref{third_term_expansion:lem} for $k=0$. 
We assume that $r^{(n)}_{K-i}$ ($i = 0,1,\dots,k-1$) has three term expansion for some $k \geq 1$. We will prove that $r^{(n)}_{K-k}$ also has three term expansion and 
we find the coefficient of $1/n^{k+3}$ in the expansion. We use (\ref{rK_k:eq}) again keeping in mind that $r^{(n)}_{K-i}$ ($i=0,1,\dots,K$) has two term expansion by 
Lemma~\ref{second_term_expansion:lem} and that $r^{(n)}_{K-k+1}$ has three term expansion according to the assumption of mathematical induction.
We have
\begin{eqnarray}
    \label{rnKkm1:eq}
	r^{(n)}_{K-k-1}   & = & \gamma^{(k+1)}_1 \frac{1}{n^{k+2}} - \gamma^{(k+1)}_2 \frac{1}{n^{k+3}} + O (\frac{1}{n^{k+4}}), \\
	\label{rnKk:eq}
	r^{(n)}_{K-k}     & = & \gamma^{(k)}_1 \frac{1}{n^{k+1}} - \gamma^{(k)}_2 \frac{1}{n^{k+2}} + O (\frac{1}{n^{k+3}}), \\
	\label{rnKkp1:eq}
	r^{(n)}_{K-k+1}   & = & \gamma^{(k-1)}_1 \frac{1}{n^{k}} - \gamma^{(k-1)}_2 \frac{1}{n^{k+1}} + \gamma^{(k-1)}_3 \frac{1}{n^{k+2}} + O (\frac{1}{n^{k+3}}).
\end{eqnarray}
The last term in (\ref{rK_k:eq}) is expanded in terms of $1/n$ as follows.
\begin{eqnarray*}
	\frac{\widetilde{r}^{(n+1)}_{K-k}}{\mu} & = & (n+1)r r^{(n+1)}_{K-k-1} + (n+1) \bar{r} r^{(n+1)}_{K-k}\\ 
                                            & = &  \bar{r} \gamma^{(k)}_1 \frac{1}{(n+1)^k} + \left( r \gamma^{(k+1)}_1 - \bar{r} \gamma^{(k)}_2 \right) \frac{1}{(n+1)^{k+1}} 
	                                           + O(\frac{1}{(n+1)^{k+2}}) \\
	                                        & = &  \bar{r} \gamma^{(k)}_1 \frac{1}{n^k}  (1 + \frac{1}{n})^{-k} + \left( r \gamma^{(k+1)}_1 - \bar{r} \gamma^{(k)}_2 \right) \frac{1}{n^{k+1}}  (1+ \frac{1}{n})^{-(k+1)}  + O(\frac{1}{(n+1)^{k+2}}) \\		
											& = &  \bar{r} \gamma^{(k)}_1 \frac{1}{n^k} + \left( r \gamma^{(k+1)}_1 - \bar{r} \gamma^{(k)}_2 - \bar{r} \gamma^{(k)}_1 k \right) \frac{1}{n^{k+1}} + O(\frac{1}{n^{k+2}}).	                                 
\end{eqnarray*}
On the other hand, we also have 
\[
	\frac{r^{(n)}_K}{n}  =  \gamma^{(0)}_1 \frac{1}{n^2} - \gamma^{(0)}_2 \frac{1}{n^3} + O (\frac{1}{n^4}).
\]
Therefore, we have 
\begin{eqnarray}
	\frac{\widetilde{r}^{(n+1)}_{K-k} r^{(n)}_K}{n \mu} & = & \gamma^{(0)}_1 \bar{r} \gamma^{(k)}_1 \frac{1}{n^{k+2}}  \nonumber \\ 
                                                        \label{complex:term}
														&  & \mbox{} +  \left[ \gamma^{(0)}_1 (r \gamma^{(k+1)}_1 - \bar{r} \gamma^{(k)}_2 - \bar{r} \gamma^{(k)}_1 k) - \gamma^{(0)}_2 \bar{r} \gamma^{(k)}_1 \right] 
															 \frac{1}{n^{k+3}}  + O (\frac{1}{n^{k+4}}). \qquad
\end{eqnarray}
From (\ref{rnK:eq}), (\ref{rnKkm1:eq})--(\ref{complex:term}), we obtain the follow expression.
\[
	r^{(n)}_{K-k} = \gamma^{(k)}_1 \frac{1}{n^{k+1}} - \gamma^{(k)}_2 \frac{1}{n^{k+2}} + \gamma^{(k)}_3 \frac{1}{n^{k+3}} + O (\frac{1}{n^{k+4}}),
\]
where
\begin{eqnarray*}
	\gamma^{(k)}_3 & = & \frac{\nu_{K-k+1}}{\mu} \gamma^{(k-1)}_3 + \left( \frac{\lambda}{\mu} + \gamma^{(0)}_1 r \right) \gamma^{(k+1)}_1 + \left( \frac{\lambda + \nu_{K-k}}{\mu} - \gamma^{(0)}_1 \bar{r} \right) \gamma^{(k)}_2  - (\gamma^{(0)}_1 k  + \gamma^{(0)}_2) \bar{r} \gamma^{(k)}_1,
\end{eqnarray*}
which is consistent with the announced result.
\end{proof}

Repeating these processes, we are also able to derive the following generalized results for $m$-term expansion ($m \geq 4$).
\subsubsection{General expansion}

Let ${(\phi)}_n$ $(-\infty < \phi < \infty$, $n \in \bbZ_+)$ denotes the Pochhammer symbol defined by
\[
{(\phi)}_n = 
	\left \{
		\begin{array}{ll}
		1, & \quad n=0, \\
		\phi(\phi+1)\cdots(\phi+n-1), & \quad n \in \bbN. \nonumber
		\end{array}
		\right.
\]
%
%Furthermore, we use the convention that $i! = 1$ if $i \leq 0$.

\begin{thm}\label{main_nonpersistent:theo}
For $m \geq 3$, we have 
\begin{equation}\label{m_terms:expansion}
r^{(n)}_{K-k} = \sum_{i=1}^m \gamma^{(k)}_i (-1)^{i+1} \frac{1}{n^{k+i}} + O (\frac{1}{n^{k+m+1}}), 
\end{equation}
where $\gamma^{(k)}_m$ is recursively defined as follows
\begin{eqnarray*}
\gamma^{(0)}_m & = & \frac{1}{\bar{r} + r \bar{q}} \sum_{k=1}^K \gamma^{(k)}_{m-k} (-1)^{k+1}, \\
\gamma^{(k)}_m & = & \frac{\nu_{K-k+1}}{\mu} \gamma^{(k-1)}_m + \frac{\lambda}{\mu} \gamma^{(k+1)}_{m-2} + \frac{\lambda + \nu_{K-k}}{\mu} \gamma^{(k)}_{m-1} \\ 						
			   &   & \mbox{} + \sum_{j=0}^{m-2} \varphi^{(k)}_j \gamma^{(0)}_{m-j-1} (-1)^{m-j}, \qquad k = 1,2,\dots,K.
\end{eqnarray*}
Furthermore, $\varphi^{(k)}_j$ is defined by
\[
	\varphi^{(k)}_j = \left \{ 
	\begin{array}{ll}
	\bar{r} \beta^{(k)}_0, & \quad j=0 \\
	r \alpha^{(k)}_j + \bar{r} \beta^{(k)}_j, & \quad j \geq 1,
	\end{array}	
	\right.
\]
where 
\begin{eqnarray*}
   \alpha^{(k)}_j & = & \sum_{i=1}^{j} \gamma^{(k+1)}_i (-1)^{j+1} \frac{(k+j)_{j-i}}{(j-i)!}, \qquad j \in \bbN, \\
   \beta^{(k)}_j  & = & \sum_{i=1}^{j+1}  \gamma^{(k)}_i (-1)^j \frac{(k+i-1)_{j+1-i}}{(j+1-i)!}, \qquad j \in \bbZ_+.
\end{eqnarray*}
\end{thm}
\begin{proof}
Theorem~\ref{main_nonpersistent:theo} is also proved using the same mathematical induction methodology as used for the cases $m=1,2$ and 3 in 
Lemmas~\ref{first_term:lem}, \ref{second_term_expansion:lem} and \ref{third_term_expansion:lem}. Another key for the proof is the following 
expansion:
\begin{equation}\label{taylor:expand}
	\left( \frac{n}{n+1} \right)^a = \left( 1 + \frac{1}{n} \right)^{-a} = \sum_{j=0}^\infty \frac{(a)_j}{j!} (-1)^j \frac{1}{n^j}, \qquad a >0.
\end{equation}

Indeed, assuming that we have $m-1$ term expansion for $r^{(n)}_{K-k}$ ($k=0,1,\dots,K$). This assumption and Proposition~\ref{sum:coro} implies announced result 
for $r^{(n)}_K$. Assuming that $r^{(n)}_{K-i}$ ($i=0,1,\dots,k-1$) already has $m$ term expansion, we prove that $r^{(n)}_{K-k}$ does too. Using the assumption 
of mathematical induction, we have
\begin{eqnarray}
    \label{rnKkm1:eqd}
	r^{(n)}_{K-k-1}   & = & \sum_{i=1}^{m-1} \gamma^{(k+1)}_i (-1)^{i+1} \frac{1}{n^{k+1+i}} + O(\frac{1}{n^{k+m+1}}) , \\
	\label{rnKk:eqd}
	r^{(n)}_{K-k}     & = & \sum_{i=1}^{m-1} \gamma^{(k)}_i (-1)^{(i+1)} \frac{1}{n^{k+i}} + O (\frac{1}{n^{k+m}} ), \\
	\label{rnKkp1:eqd}
	r^{(n)}_{K-k+1}   & = & \sum_{i=1}^m \gamma^{(k-1)}_i (-1)^{(i+1)} \frac{1}{n^{k-1+i}} + O (\frac{1}{n^{k+m}} ), \\
	\label{rn1Kkm1:eqd}
	(n+1) r^{(n+1)}_{K-k-1} & = & \sum_{i=1}^{m-1} \gamma^{(k)}_i (-1)^{i+1} \frac{1}{n^{k+i}} \left( 1 + \frac{1}{n} \right)^{-(k+i)} + O (\frac{1}{n^{k+m}} ), \\
	\label{rn1Kkm:eqd}
	(n+1) r^{(n+1)}_{K-k} & = & \sum_{i=1}^{m-1} \gamma^{(k)}_i (-1)^{i+1} \frac{1}{n^{k+i-1}} \left( 1 + \frac{1}{n} \right)^{-(k+i-1)} + O (\frac{1}{n^{k+m-1}} ), \\
	\label{rnK:eqd}
	r^{(n)}_K & = & \sum_{i=1}^{m-1} \gamma^{(0)}_i \frac{1}{n^i} + O(\frac{1}{n^m}).
\end{eqnarray}
It should be noted that $r^{(n)}_{K-k+1}$ has $m$ term expansion (\ref{rnKkp1:eqd}) due to mathematical induction. 
We further expand (\ref{rn1Kkm1:eqd}) and (\ref{rn1Kkm:eqd}) in terms of $1/n$ using (\ref{taylor:expand}). Finally, we substitute these expansions, 
(\ref{rnKkm1:eqd}), (\ref{rnKk:eqd}), (\ref{rnKkp1:eqd}) and (\ref{rnK:eqd}) into (\ref{rK_k:eq}) and subtract the coefficient of $1/n^{k+m}$ in order to obtain $\gamma^{(k)}_m$ 
as announced in (\ref{m_terms:expansion}).
\end{proof}

\begin{prop}\label{asymptotic:theo}
We have 
\[
	C^{(0)}_1  \frac{1}{n!} (\gamma^{(0)}_1)^n  n^{-\frac{\nu_K}{\mu(\bar{r} + r\bar{q})}} \leq \pi_{n,K} \leq C^{(0)}_2  \frac{1}{n!} (\gamma^{(0)}_1)^n  n^{-\frac{\nu_K}{\mu(\bar{r} + r\bar{q})}},
\]
where $C^{(0)}_1$ and $C^{(0)}_2$ are positive numbers independent of $n$.
\end{prop}
\begin{proof}
We rewrite (\ref{rnK:eq}) by expressing $\gamma^{(0)}_1, \gamma^{(0)}_2$ and $\gamma^{(0)}_3$ in terms of given parameters as 
\[
	r^{(n)}_K = \frac{\lambda p }{\mu(\bar{r} + r \bar{q})n} \left[ 1 - \frac{\nu_K}{\mu(\bar{r}+r\bar{q})} \frac{1}{n} 
                + \frac{\nu_K^2}{\mu^2(\bar{r}+r\bar{q})^2}  \left( 1  + \frac{\lambda (\bar{p} \bar{r} + r \bar{q}) }{\nu_K}    \right)  \frac{1}{n^2 }    + O (\frac{1}{n^3}) \right],
\]
in order to obtain the form 
\[
	r^{(n)}_K = \frac{\gamma^{(0)}_1}{n} \left[ 1 - \frac{a}{n} + \frac{b}{n^2} + O(\frac{1}{n^3}) \right],
\]
where 
\[
	a = \frac{\nu_k}{\mu (\bar{r} + r \bar{q})}, \qquad b = a^2 \left( 1 + \frac{\lambda(\bar{p} \bar{r} + r \bar{q})}{\nu_K} \right),
\]
satisfying $a^2 - 4b < 0$.

This expression is enough to guarantee the announced result according to Theorem 3.2 in Liu et al.~\cite{Binliu11}.
It should be remarked here that Liu et al.~\cite{Binliu11} derive asymptotic results for a special model where $c = K$ and $r=1$ and $\nu_i = i \nu$ ($i=0,1,\dots,c$). 
\end{proof}

\begin{coro}\label{asymptotic:coro_qr_l_1} There exist $C^{(k)}_1 > 0$ and $C^{(k)}_2 > 0$ independent of $n$ such that
\[
	C^{(k)}_1 \frac{1}{n!} (\gamma^{(0)}_1)^n  n^{-\frac{\nu_K}{\mu(\bar{r} + r\bar{q})} -k}  \leq 
	\pi_{n,K-k}   \leq 
	C^{(k)}_2 \frac{1}{n!} (\gamma^{(0)}_1)^n  n^{-\frac{\nu_K}{\mu(\bar{r} + r\bar{q})} -k}, \qquad n \to \infty,
\]
for $k=1,2,\dots,K$.
%as $n \to \infty$.
\end{coro}
\begin{proof}
We have
\[
	\vc{\pi}_n = \pi_{K,n-1} \vc{r}^{(n)}, 
\]
leading to 
\[
	\frac{\pi_{K-k,n}}{\pi_{K,n}} = \frac{r^{(n)}_{K-k}}{r^{(n)}_{K}}.
\]
Thus, Proposition~\ref{asymptotic:theo} and Lemma~\ref{first_term:lem} imply the announced result.
\end{proof}

\begin{remark}
In our model with two type of nonpersistent customers, a customer in the orbit leaves the system with probability  
$\bar{r} + r \bar{q}$. Thus, we may think that this model is equivalent to the corresponding one with 
one type of nonpersistent customer~\cite{Binliu11}, where a blocked retrial customer abandons with probability $\bar{r} + r \bar{q}$. 
This observation is confirmed in the asymptotic results presented in Proposition~\ref{asymptotic:theo} and Corollary~\ref{asymptotic:coro_qr_l_1} 
because the formulae here involve only $\bar{r} + r \bar{q}$. 
However, we see that the third term in the expansion of $r^{(n)}_K$ involves not only $\bar{r} + r \bar{q}$ but also 
$p$, $q$ and $r$, individually.  We further numerically investigate this matter again in Section~\ref{numerical:exam}. 
\end{remark}

\subsection{The case $q=r=1$}\label{second_case:sec}

\subsubsection{Explicit results for the cases $K=1$ and $K = 2$}
In this section, we present explicit expressions for $r^{(n)}_k$ ($k=0,1,\dots,K$) for the cases $K=1$ and $K=2$. 
\begin{thm}
For the cases $K=1$ and $K=2$, $r^{(n)}_k$ ($k=0,1,\dots,K$) is given as follows.
\begin{itemize}
\item{$K=1$,}
\begin{align*}
r^{(n)}_0 & =  \frac{\lambda p } {n \mu}, \\
r^{(n)}_1 & =  \frac{\lambda p } {n \mu} \frac{\lambda + n \mu}{\nu_1}.
\end{align*}
\item{$K=2$,}
\begin{align*}
r^{(n)}_0 & = \frac{\lambda p \nu_1}{n \mu (\lambda + \nu_1 + n \mu)}, \\
r^{(n)}_1 & = \frac{\lambda p (\lambda + n \mu)}{n \mu (\lambda + \nu_1 + n\mu)}, \\
r^{(n)}_2 & = \frac{\lambda p \left[ (\lambda + n\mu)^2 + n \mu \nu_1  \right]   }{n \mu \left[ \nu_2 ( \lambda + \nu_1 + (n+1) \mu ) + \lambda p \nu_1  \right]} \frac{\lambda + \nu_1 + (n+1) \mu}{\lambda + \nu_1 + n \mu}.
\end{align*}
\end{itemize}
\end{thm}
\begin{proof}
In both cases $K=1$ and $K= 2$, $r^{(n)}_0$ and $r^{(n)}_1$ are explicitly obtained using Proposition~\ref{sum:coro} and equation (\ref{i0:eq}). For the case $K=2$, 
equation (\ref{ic:eq}) yields 
\[
	r^{(n)}_2 = \frac{\lambda p + \lambda r^{(n)}_1}{\lambda p + \nu_2 - (n+1)\mu r^{(n+1)}_1}.
\]
Substituting the explicit expression of $r^{(n)}_1$ into this equation yields the announced result.
\end{proof}

\subsubsection{Taylor series expansion}

In what follows, we use the sequences $\{ \theta^{(k)}_n; n \in \bbZ_+ \}$ where $k$ implies the number of idle servers with the convention that $\theta^{(k)}_n = 0$ if $k > K$. 
It should be noted that the behavior of the system for the case of $q=r=1$ is totally different from that for the case $q<1$ or $r<1$. The order of 
$r^{(n)}_{K-k}$ is $1/n^{k+1}$ in the former while we will show that the order of $r^{(n)}_{K-k}$ in latter case is $1/n^k$. 
The methodology in this section is almost the same as that of Section~\ref{first_case:sec} except for the expansion of $r^{(n)}_K$. 

\begin{lem}\label{first_term_expandqr1:lem}
We have $\lim_{n \to \infty} n^k r^{(n)}_{i} = 0$ ($i= 0,1,\dots,K-k-1$) and 
\[
	r^{(n)}_{K-k} = \theta^{(k)}_0 \frac{1}{n^k} + o(\frac{1}{n^k}), \qquad k = 1,2,\dots,K.
\]
where 
\[
	\theta^{(1)}_0 = \frac{\lambda p}{\mu}, \qquad \theta^{(k)}_0 = \frac{\nu_{K-k+1}}{\mu} \theta^{(k-1)}_0, \qquad k = 2,3,\dots,K.
\]
\end{lem}
\begin{proof}

Proposition~\ref{sum:coro} becomes 
\begin{equation}\label{sumr:refined_eq}
	\sum_{k=1}^{K} r^{(n)}_{K-k} = \frac{\lambda p}{n\mu}, 
\end{equation}
from which we have $\lim_{n \to \infty} r^{(n)}_{K-k} = 0$ and $n \mu r^{(n)}_{K-k} \leq \lambda p$ ($k=1,2,\dots,K$). 

Furthermore, 
\[
	(\lambda p + \nu_K) r^{(n)}_K = \lambda r^{(n)}_{K-1} + \lambda p + (n+1) \mu r^{(n+1)}_{K-1} r^{(n)}_K < \lambda r^{(n)}_{K-1} + \lambda p + \lambda p r^{(n)}_K,
\]
where we have used $(n+1)\mu r^{(n+1)}_{K-1} \leq \lambda p$. 
Thus, we have 
\[
	\nu_K r^{(n)}_K \leq \lambda r^{(n)}_{K-1} + \lambda p 
\]
leading to the fact that $r^{(n)}_K$ is bounded because $\lim_{n \to \infty} r^{(n)}_{K-1} = 0$. This fact, (\ref{i0:eq}) and (\ref{i1c_m_1:eq}) 
show that $\lim_{n \to \infty} n r^{(n)}_i = 0$ ($i=0,1,\dots,K-1$). This and (\ref{sumr:refined_eq}) lead to 
\[
	r^{(n)}_{K-1} = \theta^{(1)}_0 \frac{1}{n} + o(\frac{1}{n}),
\]
where $\theta^{(1)}_0 = (\lambda p) / \mu$.

We prove Lemma~\ref{first_term_expandqr1:lem} by mathematical induction. Indeed, Lemma~\ref{first_term_expandqr1:lem} is true for $k=1$.
Assuming that Lemma~\ref{first_term_expandqr1:lem} is true for $k-1$, i.e.,
$\lim_{n \to \infty} n^{k-1} r^{(n)}_{i} = 0$ ($i= 0,1,\dots,K-k$) and 
\[
	r^{(n)}_{K-k+1} = \theta^{(k-1)}_0 \frac{1}{n^{k-1}} + o(\frac{1}{n^{k-1}}).
\]
for some $k \geq 2$. We will prove that Lemma~\ref{first_term_expandqr1:lem} is true for $r^{(n)}_{K-k}$.
%{ic:eq}{i1c_m_1:eq}

Multiplying (\ref{i0:eq}) by $n^{k-1}$ and taking the limit yields $ \lim_{n \to \infty} n^{k} r^{(n)}_0 = 0$. 
Next, multiplying $n^{k-1}$ by (\ref{i1c_m_1:eq}) with $i=1$ and using $ \lim_{n \to \infty} n^{k} r^{(n)}_0 = 0$, 
we obtain $\lim_{n \to \infty} n^{k} r^{(n)}_1 = 0$. By repeating this process we can 
show that 
\[
	\lim_{n \to \infty} n^k r^{(n)}_{i} = 0, \qquad i=0,1,\dots,K-k-1.
\]
 Furthermore, (\ref{rK_k:eq}) is simplified to  

\begin{equation}\label{rn_Kmk:eq}
	r^{(n)}_{K-k} = \frac{\lambda}{n \mu} r^{(n)}_{K-k-1} + \frac{\nu_{K-k+1}}{n \mu} r^{(n)}_{K-k+1} - \frac{\lambda + \nu_{K-k}}{n \mu} r^{(n)}_{K-k} + \frac{(n+1) r^{(n+1)}_{K-k-1} r^{(n)}_K }{n }. 
\end{equation}

According to the assumptions of mathematical induction, we have
\begin{align*}
\frac{\lambda}{ n \mu} r^{(n)}_{K-k-1} & = \frac{\lambda}{ n^{k+1} \mu} n^k r^{(n)}_{K-k-1} = o(\frac{1}{n^{k+1}}), \\
\frac{\lambda + \nu_{K-k}}{n \mu} r^{(n)}_{K-k} & = \frac{\lambda + \nu_{K-k}}{n^k \mu} n^{k-1} r^{(n)}_{K-k} = o(\frac{1}{n^{k}}), \\
\frac{(n+1) r^{(n)}_{K-k-1} r^{(n)}_K }{n} & = \frac{n+1}{n}  \left( r^{(n)}_{K-k-1} n^k \right) r^{(n)}_K  \frac{1}{n^k} = o(\frac{1}{n^k}), \\
\frac{\nu_{K-k+1}}{n \mu} r^{(n)}_{K-k+1} & = \frac{\nu_{K-k+1}}{\mu}  \theta^{(k-1)}_0  \frac{1}{n^k} + o(\frac{1}{n^k}).
\end{align*}
Substituting these formulae into (\ref{rn_Kmk:eq}) leads to 
\[
	r^{(n)}_{K-k} = \frac{\nu_{K-k+1}}{ \mu } \theta^{(k-1)}_0 \frac{1}{n^k} + o (\frac{1}{n^k}),
\]
implying the announced result.

Finally, we separately derive a one term expansion for $r^{(n)}_K$. Equation (\ref{ic:eq}) is simplified to 
\[
    \lambda r^{(n)}_{K-1}  +  \left( -\lambda p - \nu_K + (n+1) \mu r^{(n+1)}_{K-1} \right) r^{(n)}_K   =  - \lambda p,
\]
or equivalently 
\begin{equation}\label{rnKqr1:eq}
	r^{(n)}_K = \frac{\lambda p}{\nu_K} + \frac{\lambda}{\nu_K}  r^{(n)}_{K-1} + \left( -\lambda p + (n+1) \mu r^{(n+1)}_{K-1} \right) r^{(n)}_K,
\end{equation}
which implies 
\[
	r^{(n)}_K = \theta^{(0)}_0 + o(1), 
\]
because $\lim_{n \to \infty} (-\lambda p + (n+1) \mu r^{(n+1)}_{K-1}) = 0$, where $\theta^{(0)}_0 = (\lambda p) / \nu_K$. 
\end{proof}

\begin{remark}
Our derivations in Section~\ref{second_case:sec} are also based on mathematical induction which has two steps. The first step is the formula for $r^{(n)}_{K-1}$. The second step is to derive the formula for $r^{(n)}_{K-k}$ provided that for $r^{(n)}_{K-i}$ ($i=0,1,\dots,k-1$) is true.
Proposition~\ref{sum:coro} is important for the first step while formula (\ref{rn_Kmk:eq}) is the key for the second step. We derive the formula for $r^{(n)}_K$ separately by (\ref{rnKqr1:eq}) after having formulae for $r^{(n)}_{K-k}$ ($k=1,2,\dots,K$).
\end{remark}

Lemma~\ref{first_term_expandqr1:lem} can be refined as follows.

\begin{lem}\label{first_term_expandqr1:lem_refined}
We have the following result
\[
    r^{(n)}_{K-k}  =  \theta^{(k)}_0 \frac{1}{n^k} + O (\frac{1}{n^{k+1}}), \qquad k \in \bbZ_+.
\] 
\end{lem}
\begin{proof}
We again prove by mathematical induction.
First, we prove Lemma~\ref{first_term_expandqr1:lem_refined} for $k=1$. Indeed, we have
\[
	r^{(n)}_{K-1} = \frac{\lambda p}{n \mu} - \sum_{k=2}^{K} r^{(n)}_{K-k}.
\]
This and Lemma~\ref{first_term_expandqr1:lem} yield the announced result for the case of $k=1$.
Assuming that for some $k \geq 2$, we have 
\[
	r^{(n)}_{K-i} = \theta^{(i)}_0 \frac{1}{n^i} + O(\frac{1}{n^{i+1}}), \qquad i = 1,2,\dots,k-1,
\]
we prove that Lemma~\ref{first_term_expandqr1:lem_refined} is also true for $r^{(n)}_{K-k}$.

Substituting $r^{(n)}_{K-k+1}$ by the assumption of mathematical induction and $r^{(n)}_{K-i}$ ($i=k+1, k, 0$) by Lemma~\ref{first_term_expandqr1:lem} into (\ref{rn_Kmk:eq}) 
and arranging the result yields Lemma~\ref{first_term_expandqr1:lem_refined} for $r^{(n)}_{K-k}$ ($k=1,2,\dots,K$). 

We can find a one term expansion for $r^{(n)}_K$ as follows. Equation (\ref{ic:eq}) is equivalent to 
\begin{eqnarray*}
	\nu_K r^{(n)}_K & = &  \lambda p + (-\lambda p + (n+1)\mu r^{(n+1)}_{K-1}) r^{(n)}_K, \\
                    & = &  \lambda p + O(\frac{1}{n}) (\theta^{(0)}_0 + o(1)), 
\end{eqnarray*}
implying 
\[
	r^{(n)}_K = \theta^{(0)}_0 + O(\frac{1}{n}).
\]

\end{proof}

We can extend the result for $m+1$ term expansion as follows. 

\begin{thm}\label{general_expand_qr1:theo}
We have 
\[
	r^{(n)}_{K-k} = \sum_{i=0}^m \theta^{(k)}_i (-1)^i \frac{1}{n^{k+i}} + O (\frac{1}{n^{k+m+1}}), \qquad m \in \bbN,
\]
where $\theta^{(k)}_i$ is recursively defined as follows. 
\begin{eqnarray*}
    \theta^{(1)}_m & = & \sum_{i=2}^{\min(K,m+1)} \theta^{(i)}_{m+1-i} (-1)^i, \\
	\theta^{(k)}_m & = & \frac{\nu_{K-k+1}}{\mu} \theta^{(k-1)}_m + \frac{\lambda}{\mu} \theta^{(k+1)}_{m-2} + \frac{\lambda + \nu_{K-k}}{\mu} \theta^{(k)}_{m-1} \nonumber \\
	               &   & \mbox{} + \sum_{j=0}^{m-1} \Phi^{(k)}_j \theta^{(0)}_{m-j-1} (-1)^{j+1}, \qquad k = 2,3,\dots,K,
\end{eqnarray*} 
where 
\[
	\Phi^{(k)}_j = \sum_{i=0}^j \theta^{(k+1)}_i (-1)^j \frac{(k+i)_{j-i}}{(j-i)!}. 
\]
Furthermore, we have
\[
	\theta^{(0)}_m = - \frac{\lambda}{\nu_K} \theta^{(1)}_{m-1} +   \frac{\mu}{\nu_K} \sum_{j=1}^m \widetilde{\Phi}^{(0)}_j \theta^{(0)}_{m-j} (-1)^j,
\]
where 
\[
	\widetilde{\Phi}^{(0)}_j = \sum_{i=1}^j \theta^{(1)}_i \frac{(i)_{j-i}}{(j-i)!} (-1)^j, \qquad j = 1,2,\dots,m.
\]
\end{thm}
\begin{proof}
We prove Theorem~\ref{general_expand_qr1:theo} using mathematical induction. 

First, we prove that Theorem~\ref{general_expand_qr1:theo} is true for $m=1$ also by mathematical induction.
We have 
\[
	r^{(n)}_{K-1} = \frac{\lambda p}{n \mu} - \sum_{i=2}^{K} r^{(n)}_{K-i}.
\]
Substituting $r^{(n)}_{K-i}$ by Lemma~\ref{first_term_expandqr1:lem_refined} into the right hand side of the above equation 
yields 
\[
	r^{(n)}_{K-1} = \theta^{(1)}_0 \frac{1}{n} - \theta^{(1)}_1 \frac{1}{n^2} + O (\frac{1}{n^3}),
\]
where $\theta^{(1)}_1 = \theta^{(2)}_0$. For some $k \geq 2$, assuming that $r^{(n)}_{K-i}$ ($i=1,2,\dots,k-1$) has 
two term expansion. We prove that $r^{(n)}_{K-i}$ also has two term expansion. Indeed, we have 

\begin{eqnarray*}
r^{(n)}_{K-k-1} & = & \theta^{(k+1)}_0 \frac{1}{n^{k+1}} + O(\frac{1}{n^{k+2}}), \\
r^{(n)}_{K-k}   & = & \theta^{(k)}_0 \frac{1}{n^k} + O (\frac{1}{n^{k+1}}), \\
r^{(n)}_{K-k+1}   & = & \theta^{(k-1)}_0 \frac{1}{n^{k-1}} - \theta^{(k-1)}_1 \frac{1}{n^{k}} + O (\frac{1}{n^{k+1}}), \\
(n+1)r^{(n+1)}_{K-k-1} & = & \theta^{(k+1)}_0 \frac{1}{(n+1)^{k}} + O(\frac{1}{n^{k+1}}) = \theta^{(k+1)}_0 \frac{1}{n^{k}} + O(\frac{1}{n^{k+1}}), \\
\frac{r^{(n)}_K}{n} & = & \theta^{(0)}_0 \frac{1}{n} + O(\frac{1}{n^2}).  
\end{eqnarray*}
Substituting these quantities into (\ref{rn_Kmk:eq}) yields the announced result for $r^{(n)}_{K-k}$.

We obtain two term expansion for $r^{(n)}_K$ as follows.
\begin{eqnarray*}
	\nu_K r^{(n)}_K & = &  \lambda p + \lambda r^{(n)}_{K-1} +  (-\lambda p + (n+1)\mu r^{(n+1)}_{K-1}) r^{(n)}_K, \\
                    & = &  \lambda p + \lambda \theta^{(1)}_0 \frac{1}{n} -\theta^{(1)}_1 \frac{1}{n^2} + O(\frac{1}{n^3}) + \left( - \mu \theta^{(1)}_1 \frac{1}{n} + O(\frac{1}{n^2}) \right) 
							\left( \theta^{(0)}_0 + O(\frac{1}{n}) \right), \\
                    & = &  \lambda p + \lambda \theta^{(1)}_0 \frac{1}{n} 
							- \mu \theta^{(0)}_0 \theta^{(1)}_1 \frac{1}{n} + O(\frac{1}{n^2}),
\end{eqnarray*}
implying that 
\[
	r^{(n)}_K = \theta^{(0)}_0 - \theta^{(0)}_1 \frac{1}{n} + O(\frac{1}{n^2}),
\] 
where 
\begin{eqnarray*}
	\theta^{(0)}_1 & = & \frac{\mu}{\nu_K} \theta^{(0)}_0 \theta^{(1)}_1 - \frac{\lambda}{\nu_K} \theta^{(1)}_0 =  - \rho^2 \frac{ p \nu_{K-1} - \nu_K}{\mu}.
\end{eqnarray*}
Thus we have proved that Theorem~\ref{general_expand_qr1:theo} is true for $m=1$.

Next, assuming that Theorem~\ref{general_expand_qr1:theo} is true for $m$ term expansion for some $m \geq 2$, we will prove that Theorem~\ref{general_expand_qr1:theo} is 
also true for $(m+1)$ term expansion. We can prove this in a similar manner as used above for the case $m=1$. Indeed, we again use 
\begin{eqnarray*}
	r^{(n)}_{K-1} &  = & \frac{\lambda p}{n \mu} - \sum_{i=2}^{K} r^{(n)}_{K-i} \\
	              &  = & \frac{\lambda p}{n \mu} - \sum_{i=2}^{K} \sum_{j=0}^{m-1} \theta^{(i)}_j \frac{1}{n^{i+j}} (-1)^j,	  
\end{eqnarray*}
according to the assumption of $m$ term expansion. Collecting the coefficients of $1/n^{m+1}$ yields Theorem~\ref{general_expand_qr1:theo} for $k=1$. 

Assuming that $r^{(n)}_{K-i}$ ($i=1,2,\dots,k-1$) has $(m+1)$ term expansion, we prove that $r^{(n)}_{K-k}$ does too. 
Indeed, according to the assumption of mathematical induction, we have 

\begin{eqnarray*}
r^{(n)}_{K-k-1} & = & \sum_{i=0}^{m-1} \theta^{(k+1)}_i (-1)^i \frac{1}{n^{k+1+i}}  + O(\frac{1}{n^{k+m+1}}), \\
r^{(n)}_{K-k} & = & \sum_{i=0}^{m-1} \theta^{(k)}_i (-1)^i \frac{1}{n^{k+i}}  + O(\frac{1}{n^{k+m}}), \\
r^{(n)}_{K-k+1} & = & \sum_{i=0}^{m} \theta^{(k-1)}_i (-1)^i \frac{1}{n^{k-1+i}} + O(\frac{1}{n^{k+m}}), \\
r^{(n)}_K & = & \sum_{i=0}^{m-1} \theta^{(0)}_i (-1)^i \frac{1}{n^i} + O(\frac{1}{n^m}), \\
(n+1) r^{(n+1)}_{K-k-1} & = &  \sum_{i=0}^{m-1} \theta^{(k+1)}_i (-1)^i \frac{1}{(n+1)^{k+i}} +  O(\frac{1}{n^{k+m}}) \\ 
                        & = &  \sum_{i=0}^{m-1} \theta^{(k+1)}_i (-1)^i \frac{1}{n^{k+i}} (1+\frac{1}{n})^{-(k+i)} +  O(\frac{1}{n^{k+m}}) \\ 
                           & = &  \sum_{j=0}^{m-1} \Phi^{(k)}_j \frac{1}{n^{k+j}} + O(\frac{1}{n^{k+m}}), \\
\frac{r^{(n)}_K}{n} & = & \sum_{i=0}^{m-1} \theta^{(0)}_i (-1)^i \frac{1}{n^{i+1}} + O(\frac{1}{n^{m+1}}),
\end{eqnarray*}
where we have used Taylor series expansion (\ref{taylor:expand}) in the seventh equality. It should be noted here that $r^{(n)}_{K-k+1}$ has $(m+1)$ term expansion according to the assumption of mathematical induction.
Substituting these formulae into (\ref{rn_Kmk:eq}) and collecting the coefficient of $1/n^{k+m}$ yields the announced 
$(m+1)$ term expansion for $r^{(n)}_{K-k}$. 

Finally, we derive $(m+1)$ term expansion for $r^{(n)}_K$.
We have 
\begin{equation}\label{rnK_trans:eq}
	r^{(n)}_K = \theta^{(0)}_0 + \frac{\lambda}{\nu_K} r^{(n)}_{K-1} + \frac{(-\lambda p + (n+1)\mu r^{(n+1)}_{K-1})r^{(n)}_K}{\nu_K}
\end{equation}
Furthermore, we have 
\begin{eqnarray*}
-\lambda p + (n+1)\mu r^{(n+1)}_{K-1} & = &  \sum_{i=1}^{m} \theta^{(1)}_i (-1)^i \frac{1}{(n+1)^{i}} +  O(\frac{1}{n^{m+1}}), \\
                                      & = &  \sum_{j=1}^{m} \widetilde{\Phi}^{(0)}_j \frac{1}{n^i} +  O(\frac{1}{n^{m+1}}),             
\end{eqnarray*}
where 
\[
	\widetilde{\Phi}^{(0)}_j = \sum_{i=1}^j \theta^{(1)}_i \frac{(i)_{j-i}}{(j-i)!} (-1)^j, \qquad j=1,2,\dots,m,
\]
and 
\[
	r^{(n)}_K    =  \sum_{i=0}^{m-1} \theta^{(0)}_i (-1)^i \frac{1}{n^i} + O(\frac{1}{n^m}).
\]
Substituting these formulae into (\ref{rnK_trans:eq}) and collecting the coefficients of $1/n^m$ yields the announced result.
\end{proof}

Finally, we obtain the following asymptotic results.
\begin{prop}\label{asymptotic_qr1:theo}
There exist two positive coefficients $D_1^{(0)}$ and $D_2^{(0)}$ independent of $n$ such that
\[
	D_1^{(0)} n^{\alpha}  \rho^n \leq \pi_{K,n} \leq D_2^{(0)} n^{\alpha}  \rho^n, \qquad 
	n \to \infty.
\]
where $\rho = \frac{\lambda p}{\nu_K}$ and $\alpha = \frac{\rho (\nu_{K}  - p \nu_{K-1} )}{p \mu}$.
\end{prop}

\begin{proof}
Using the results derived in Lemma~\ref{first_term_expandqr1:lem_refined} and Theorem~\ref{general_expand_qr1:theo}, 
we obtain 
\[
	r^{(n)}_K = \theta^{(0)}_0 - \theta^{(0)}_1 \frac{1}{n} + \theta^{(0)}_2 \frac{1}{n^2} + O(\frac{1}{n^3}),
\]
where
%\begin{eqnarray*}
\[
  \theta^{(0)}_0  =  \rho, \qquad \theta^{(0)}_1  =  -\rho^2 \frac{\nu_K  - p \nu_{K-1} }{p \mu}.
\]
%  \theta^{(0)}_2 & = & \frac{\rho^2 \nu_{K-1}}{\mu^2} \left( \mu + \nu_{K-1} + \rho (\nu_{K-1}-\nu_{K-2}) - \frac{\nu_K}{p} \right).
%\end{eqnarray*}
%
We further transform as follows. 
\[
	r^{(n)}_K = \rho \left( 1 + \alpha \frac{1}{n} + \beta \frac{1}{n^2} + O ( \frac{1}{n^3} ) \right ),
\]
where 
\[
	\alpha = \rho \frac{\nu_K  -  p \nu_{K-1} }{\mu p}, \qquad \beta = \frac{\theta^{(0)}_2}{\rho}.
\]
Using Theorem 3.2 in Liu and Zhao~\cite{Binliu10}, we obtain the announced result.
\end{proof}

\begin{coro}\label{piKminusk:coro}
There exist $D^{(k)}_1 > 0$ and $D^{(k)}_2 >0$ ($k = 1,2,\dots,K$) independent of $n$ such that 
\[
	D^{(k)}_1 n^{\alpha-k}  \rho^n \leq \pi_{K-k,n} \leq D^{(k)}_2 n^{\alpha-k}  \rho^n, \quad n \to \infty.
\]
\end{coro}

\begin{proof}
We again use the formula 
\[
	\vc{\pi}_n = \pi_{K,n-1} \vc{r}^{(n)}, 
\]
in order to have
\[
	\frac{\pi_{K-k,n}}{\pi_{K,n}} = \frac{r^{(n)}_{K-k}}{r^{(n)}_{K}}.
\]
Thus, Proposition~\ref{asymptotic_qr1:theo} and Lemma~\ref{first_term_expandqr1:lem_refined} imply the announced result.
\end{proof}

\section{Numerical examples}\label{numerical:exam} 
In this section, we present some numerical examples to show the accuracy of the Taylor series expansions and the tail asymptotic behavior of the stationary distribution. Section~\ref{effect:rho} shows the influence of $\rho^* = \lambda/\nu_K$ on the relative errors of the rate matrices. 
The case of nonpersistent retrial customers, i.e., $q \neq 1$ or $r \neq 1$, is presented in Section~\ref{non-persistent:sec} while the case of persistent retrial customers, i.e., $q, r \neq 1$ is shown in Section~\ref{persistent:sec}.

\subsection{Effect of the traffic intensity}\label{effect:rho}
In this section, we present numerical examples to show the accuracy of Taylor series expansions. 
We consider the following parameter set: $\mu = 1$, $K = c = 5$, $r = 0.5$, $p = 0.7$, $q = 0.7$ and $\nu_i = i$ $(i=0,1,\dots,c)$.
The arrival rate $\lambda$ is calculated from the traffic intensity $\rho^* = \lambda/\nu_K$ given in Tables~\ref{case1:N100} and~\ref{case1:N1000}. Let $\vc{r}^{(N)}_k$ denote the last 
row of $R^{(N)}_k$. Using Proposition~\ref{prop:rn}, we calculate the approximation $R^{(N)}_k$ to $R^{(N)}$ where $k$ is the smallest natural number such that $||\vc{r}^{(N)}_{k} - \vc{r}^{(N)}_{k-1}|| < 10^{-10}$ and $||x|| = \sum_{0}^K |x_i|$ for $x = (x_0,x_1,\dots,x_K)$. Since the difference between $\vc{r}^{(N)}_{k-1}$ 
and $\vc{r}^{(N)}_{k}$ is small enough, we numerically consider $\vc{r}^{(N)}_{k}$ as the exact value of $\vc{r}^{(N)}$. At the same time, let $\vc{r}^{(N,1)}$, $\vc{r}^{(N,2)}$ and $\vc{r}^{(N,3)}$ denote the first, the second and the third Taylor series expansions of $\vc{r}^{(N)}$. In Table~\ref{case1:N100}  ($N=100$) and Table~\ref{case1:N1000} ($N=1000$), the second, the third and the fourth columns represent the 
relative errors, i.e., 
\[ 
\frac{||\vc{r}^{(N,1)} - \vc{r}^{(N)}_k||}{||\vc{r}^{(N)}_k||}, \quad
\frac{||\vc{r}^{(N,2)} - \vc{r}^{(N)}_k||}{||\vc{r}^{(N)}_k||},  \quad
\frac{||\vc{r}^{(N,3)} - \vc{r}^{(N)}_k||}{ ||\vc{r}^{(N)}_k||},
\]
respectively.
%Tables~\ref{case1:N100} and~\ref{case1:N1000} present numerical results for the cases $N=100$ and $1000$. 
We observe that Taylor series expansions give a fairly good accuracy, especially for the case $N=1000$. 
We also observe from Tables~\ref{case1:N100} and~\ref{case1:N1000} that the orders of the relative errors are almost insensitive to the traffic intensity $\rho^*$.

\begin{table}[htb]
\caption{Relative error of $\vc{r}^{(N)}$ for the case $\bar{r} + r \bar{q} > 0$ ($N=100$).}
\label{case1:N100}
\begin{center}
\begin{tabular}{|c|c|c|c|}
\hline
Traffic intensity ($\rho^*$)  &   First order   &   Second order  &  Third order \\
\hline
0.1     & 0.078979804     &    0.006347302  &       0.000512522 \\
\hline
0.2     & 0.078922701     &    0.006528123  &       0.000548023 \\
\hline
0.3     & 0.078865830     &    0.006708717  &       0.000584347 \\
\hline
0.4     & 0.078809192     &    0.006889085  &       0.000621491 \\
\hline
0.5     & 0.078752783     &    0.007069227  &       0.000659455 \\
\hline
0.6     & 0.078696602     &    0.007249146  &       0.000698238 \\
\hline
0.7     & 0.078640650     &    0.007428842  &       0.000737837 \\
\hline
0.8     & 0.078584923     &    0.007608316  &       0.000778252 \\
\hline
0.9     & 0.078529420     &    0.007787571  &       0.000819482 \\
\hline
\end{tabular}
\end{center}
\end{table}

\begin{table}[htb]
\caption{Relative error of $\vc{r}^{(N)}$ for the case $\bar{r} + r \bar{q} > 0$ ($N=1000$).}
\label{case1:N1000}
\begin{center}
\begin{tabular}{|c|c|c|c|}
\hline
Traffic intensity ($\rho^*$)  &   First order   &   Second order  &  Third order \\
\hline
0.1     & 0.007711805     &    0.000061185  &       0.000000491 \\
\hline 
0.2     & 0.007711190     &    0.000062962  &       0.000000525 \\
\hline
0.3     & 0.007710574     &    0.000064739  &       0.000000560 \\
\hline
0.4     & 0.007709959     &    0.000066516  &       0.000000596 \\
\hline
0.5     & 0.007709344     &    0.000068292  &       0.000000633 \\
\hline
0.6     & 0.007708729     &    0.000070068  &       0.000000671 \\
\hline
0.7     & 0.007708115     &    0.000071844  &       0.000000709 \\
\hline
0.8     & 0.007707500     &    0.000073620  &       0.000000748 \\
\hline
0.9     & 0.007706887     &    0.000075395  &       0.000000788 \\
\hline
\end{tabular}
\end{center}
\end{table}

\subsection{Nonpersistent retrial customers}\label{non-persistent:sec}
%add the information $\mu = 1$
In this section, we present numerical results for the non-persistent case, i.e., $q \neq 1$ or $r \neq 1$.
In particular, we consider the following parameter set: $p=0.7, q=0.7, r=0.5,  K =  c = 10$, $\nu_i = i$ ($i = 0,1,\dots,c$), $\mu = 1$ and $\epsilon =  10^{-10}$.
Figure~\ref{taylor_series_expand_rho050920:fig} represents the Taylor expansions and the exact value of $r^{(n)}_K$ against $n$ for the cases: $\rho^* = 0.5, 0.9$ and 2.0. We observe that the accuracy of the Taylor expansion increases with $m$ (the number of terms) and also with $n$ (the number of customers in the orbit) as is expected.

Figure~\ref{relative_error_rho050920:fig} shows the relative error or the Taylor expansions for $r^{(n)}_K$ against $n$ for the cases: $\rho^* = 0.5, 0.9$ and 2.0. 
The exact value for $r^{(n)}_K$ is calculated using Proposition~\ref{prop:rn}. In particular, we approximate $\vc{r}^{(n)}$ by $\vc{r}^{(n)}_{2^{k+1}-1}$, where $k$ is the smallest natural number such that $||\vc{r}^{(n)}_{2^{k+1}-1} - \vc{r}^{(n)}_{2^{k}-1}|| < 10^{-10}$. 
We observe that the relative errors of the Taylor expansions decrease with the number of terms and also with $n$ as expected. These observations verify the correctness of our Taylor series expansions.

We present a numerical example to show the tail asymptotic behavior for the joint stationary distribution. 
The stationary distribution is obtained using the methodology presented by Phung-Duc~\cite{phung-duc13}. We first approximate $\vc{r}^{(N)}$ ($N=300$) by 
$\vc{r}^{(n)}_{2^{k+1}-1}$, where $k$ is the smallest natural number such that $||\vc{r}^{(n)}_{2^{k+1}-1} - \vc{r}^{(n)}_{2^{k}-1}|| < 10^{-10}$. 
We then obtain an approximation to the stationary distribution by that of the censored Markov chain on the state space $\{ l(0),l(1),\dots, l(N) \}$ whose infinitesimal generator is given by (\ref{sensored:chain}). Phung-Duc~\cite{phung-duc13} presents an algorithm with the computational complexity of $O(K)$ for solving this Markov chain. 

Using the joint stationary distribution, we plot 
\[
	\frac{\pi_{K,n} n!}{(\gamma^{(0)}_1)^n} 
\]
against $n$. We consider two cases: i) $p=0.7, q=0.7, r=0.5$ (two types of nonpersistent customers) and ii) $p = 0.7, q = 0.35$ and $r=1$ (one type of nonpersistent customers) while keeping $\lambda$, $K=c=10$, $\nu_i$ ($i = 0.1,\dots,c$) and $\mu=1$ the same.
Cases i) and ii) are equivalent in the sense that the probability that a blocked retrial customer arrives at the servers again is 0.35 in both cases. We observe that the curves are asymptoticly linear. This fact agrees with Proposition~\ref{asymptotic:theo}.  
Figure~\ref{tail_asymptotic_nonpersist1_rho0509:fig} shows the curves for $\rho^* = 0.5, 0.9$ while Figure~\ref{tail_asymptotic_nonpersist1_rho2030:fig} presents the curves for $\rho^* = 2.0, 3.0$.
We observe from Figures~\ref{tail_asymptotic_nonpersist1_rho0509:fig} and \ref{tail_asymptotic_nonpersist1_rho2030:fig} that the probability for case i) is greater than that of case ii) when the number of customers in the orbit is large enough. 
The reason is that in average retrial customers in case i) stay in the orbit longer than those in case ii).

\begin{figure}[tbhp]
\begin{center}
\includegraphics[scale=0.8]{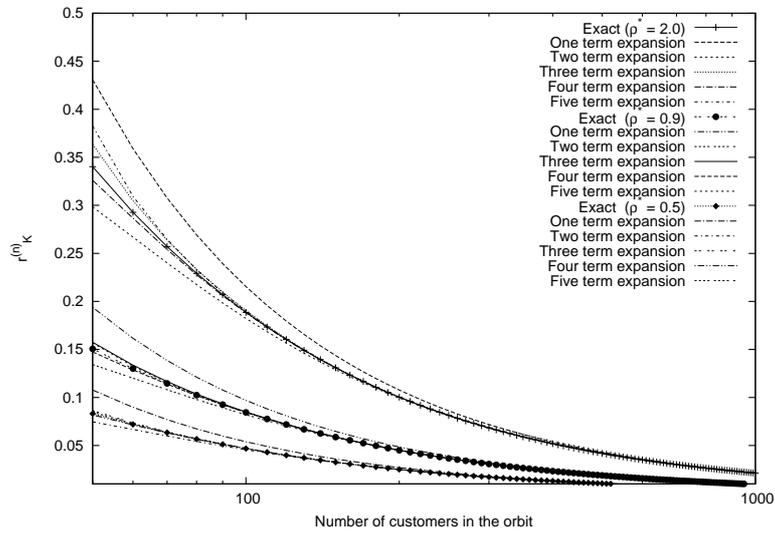} 
\caption{Taylor expansion for $r^{(n)}_K$ vs. $n$.}
\label{taylor_series_expand_rho050920:fig}
\end{center}
\end{figure}
\begin{figure}[tbhp]
\begin{center}
\includegraphics[scale=0.8]{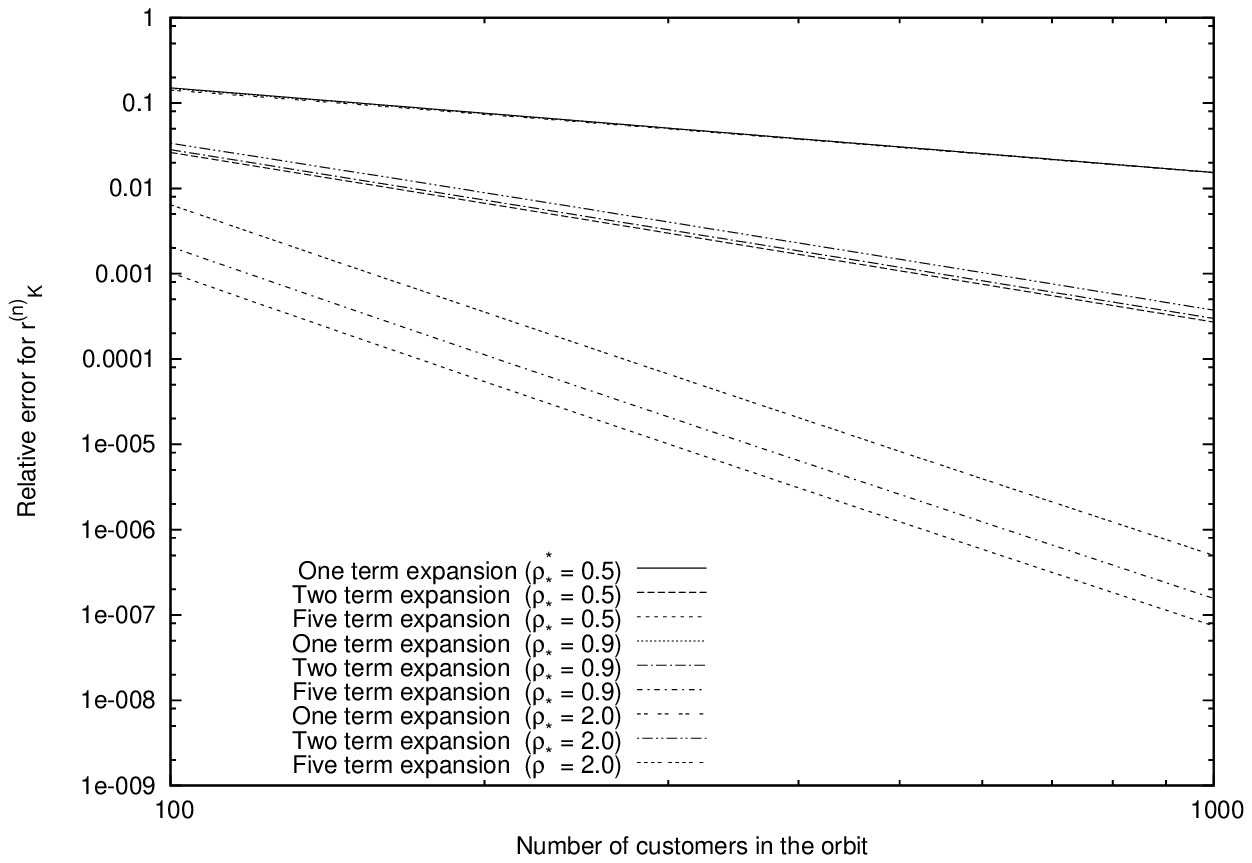} 
\caption{Relative error for $r^{(n)}_K$ vs. $n$.}
\label{relative_error_rho050920:fig}
\end{center}
\end{figure}
\begin{figure}[tbhp]
\begin{center}
\includegraphics[scale=0.8]{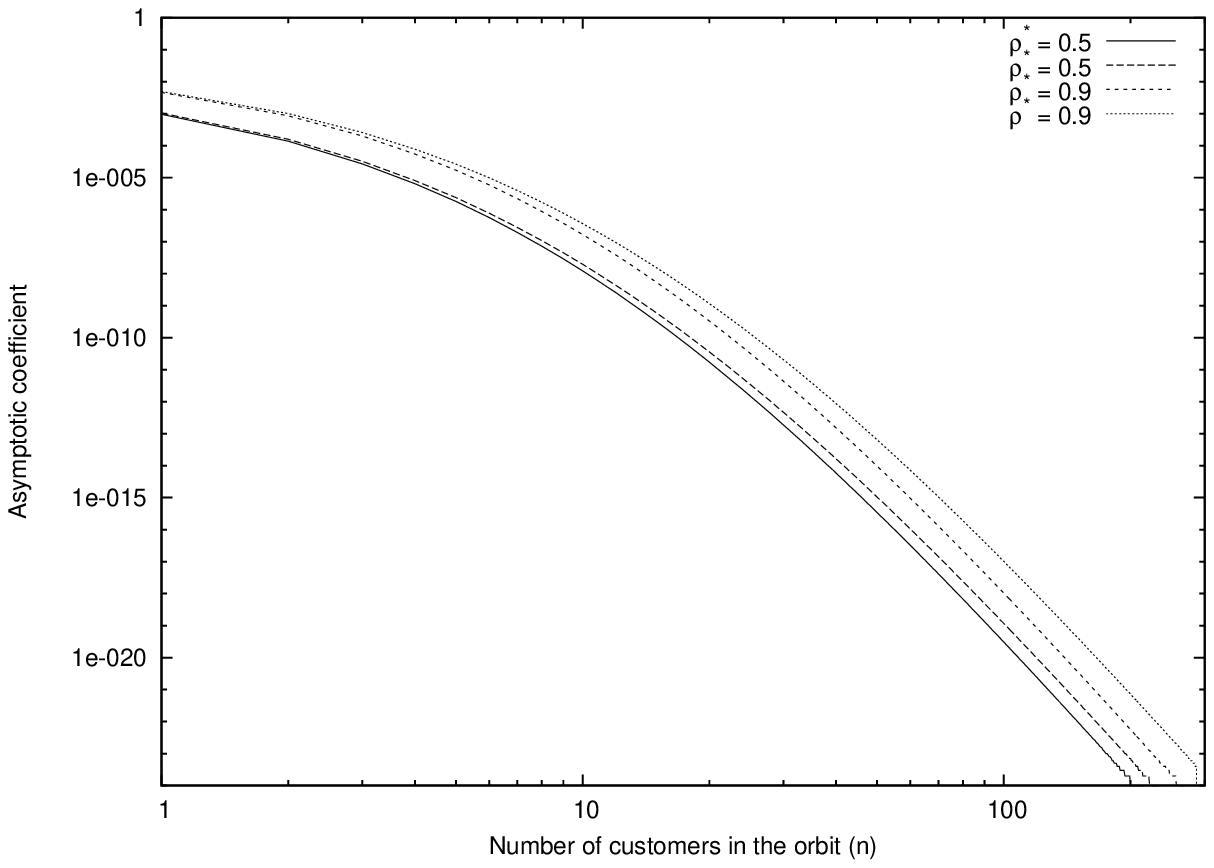} 
\caption{Coefficient of the asymptotic formula.}
\label{tail_asymptotic_nonpersist1_rho0509:fig}
\end{center}
\end{figure}
\begin{figure}[tbhp]
\begin{center}
\includegraphics[scale=0.8]{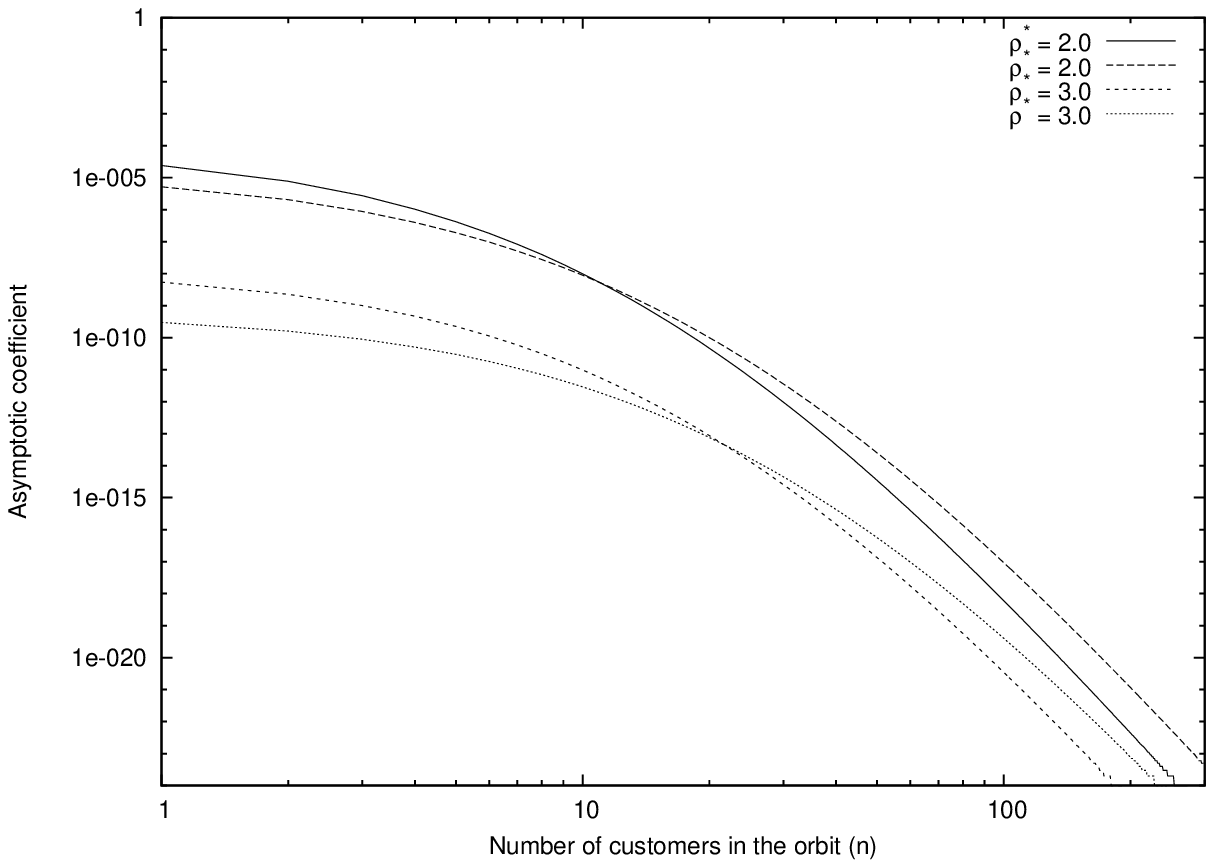} 
\caption{Coefficient of the asymptotic formula.}
\label{tail_asymptotic_nonpersist1_rho2030:fig}
\end{center}
\end{figure}

\subsection{Persistent retrial customers}\label{persistent:sec}
In this section, we consider the persistent case where customers never abandon the system, i.e., $p = q = r = 1$. 
Other parameters are given by $K =  c = 10$, $\nu_i = i$ ($i = 0,1,\dots,c$) and $\mu = 1$. Figure~\ref{taylor_series_expand_rho050709_persistent:fig} expresses the Taylor expansion formulae against the number of customers in the orbit $n$ for the cases $\rho^* = 0.7$ and 0.9. We observe that the Taylor series expansions converge fast to the exact value after a few terms. Figure~\ref{relative_error_rho050709_persist:fig} shows the relative errors against the number of customers in the orbit for the cases  $\rho^* = 0.7$ and $\rho^* = 0.9$. We also observe that the relative errors of the Taylor series expansions  decrease with the number of expansion terms and with the number of customers in the orbit $n$. These observations verify the correctness of our Taylor series expansions. 

\begin{figure}[tbhp]
\begin{center}
\includegraphics[scale=0.8]{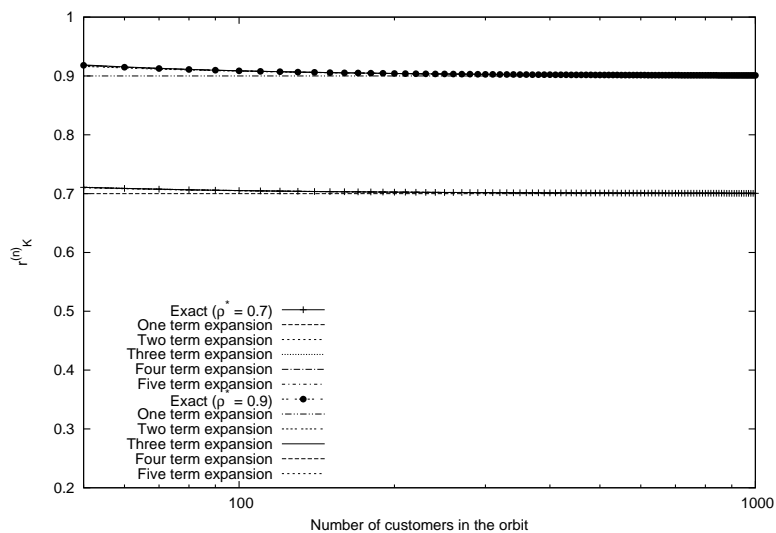} 
\caption{Taylor expansion for $r^{(n)}_K$ vs. $n$.}
\label{taylor_series_expand_rho050709_persistent:fig}
\end{center}
\end{figure}
\begin{figure}[tbhp]
\begin{center}
\includegraphics[scale=0.8]{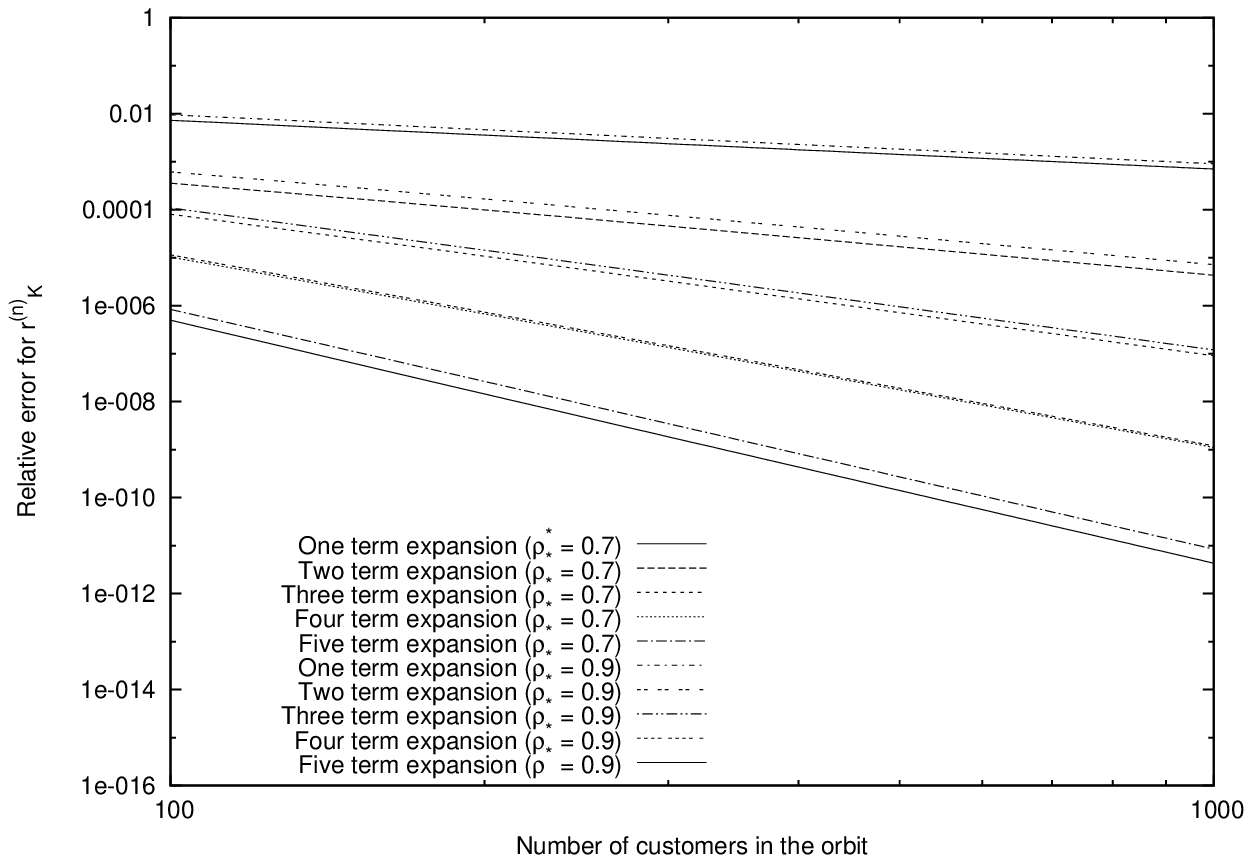} 
\caption{Relative error for $r^{(n)}_K$ vs. $n$.}
\label{relative_error_rho050709_persist:fig}
\end{center}
\end{figure}
\begin{figure}[tbhp]
\begin{center}
\includegraphics[scale=0.8]{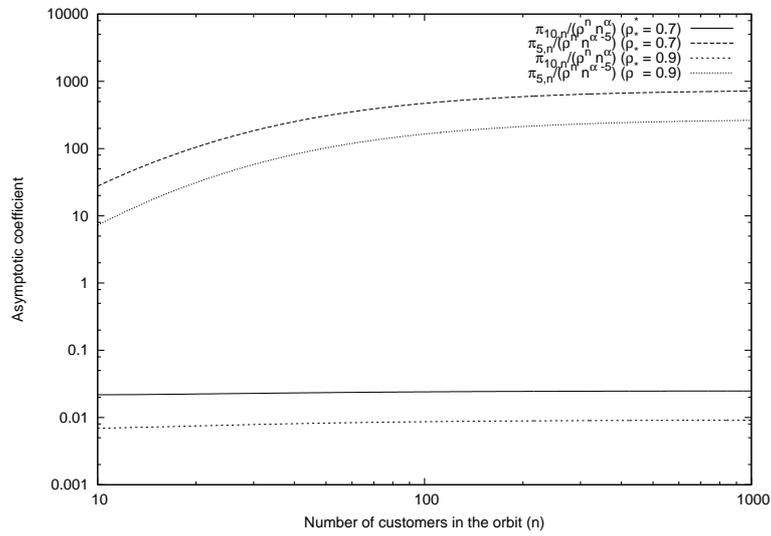} 
\caption{Coefficient of the asymptotic formula.}
\label{tail_persist_rho:fig}
\end{center}
\end{figure}

We investigate the tail probability for the stationary distribution. In particular, we verify the asymptotic formulae in Proposition~\ref{asymptotic_qr1:theo}  and Corollary~\ref{piKminusk:coro} by 
investigating the behavior of $\pi_{K-k,n}/(\rho^n n^{\alpha-k})$ ($k=0,5$) against $n$ when $n$ is large enough. To this end, we consider two cases: $\rho^* = 0.7$ and $\rho^* = 0.9$ 
while other parameters are given by $p = q = r = 1$, $K =  c = 10$, $\nu_i = i$ ($i = 0,1,\dots,c$) and $\mu = 1$. 
We observe from Figure~\ref{tail_persist_rho:fig} that $\pi_{K-k,n}/(\rho^n n^{\alpha-k})$ tends to some constant as $n \to \infty$. 
This suggests that there exists some $D_k$ such that $\lim_{n \to \infty} \pi_{K-k,n}/(\rho^n n^{\alpha-k}) = D_k$ ($k=0,1,\dots,K$). 
This fact is consistent with the theoretical result derived in Kim et al.~\cite{3Kim2012}.

\section{Concluding remarks}\label{conclusion:sec}

In this paper, using a unified perturbation approach, we have derived Taylor series expansion for any element of 
the rate matrices. We have derived recursive formulae for the coefficients for which both numerical and symbolic 
algorithms can be implemented. Using the expansion, we have been able to compute the rate matrices with any desired accuracy by a forward type algorithm. 
Furthermore, by applying the result of Liu et al.~\cite{Binliu11,Binliu10}, we have also obtained asymptotic formulae for the stationary distribution. 

The methodology developed in this paper can be applied to other Markov chains with inhomogeneous structure. 
It is easy to extend our analysis to the model with state-dependent arrivals. It is interesting to examine the methodology for level-dependent QBD with 
more dense rate matrices. This will be the topic for any future research.

\section*{Acknowledgements}
Tuan Phung-Duc was supported in part by Japan Society for the Promotion of Science, JSPS Grant-in-Aid for Young Scientists (B), Grant Number 2673001. 
%This work was supported by JSPS KAKENHI Grant Number 12345678.

\end{document}